\documentclass[10pt]{amsart}
\usepackage{amsfonts,amsthm,amsmath,color,graphics,hyperref,pict2e}
\usepackage[all]{xy}
\usepackage{indentfirst}

\newcommand{\myperm}{P}

\newlength{\halfbls}

\newtheorem{theorem}{Theorem}[section]
\newtheorem{condtheorem}[theorem]{Conditional Theorem}
\newtheorem*{NoNumberTheorem}{Theorem}
\newtheorem{conjecture}[theorem]{Conjecture}
\newtheorem*{mainconjecture}{Main Conjecture}
\newtheorem{corollary}[theorem]{Corollary}
\newtheorem{lemma}[theorem]{Lemma}
\newtheorem{proposition}[theorem]{Proposition}
\newtheorem*{NoNumberProposition}{Proposition}

\theoremstyle{remark}

\newtheorem{remark}[theorem]{Remark}
\newtheorem{problem}[theorem]{Problem}

\newtheorem{example}[theorem]{Example}

\begin{document}
\title[Curvature, Lyapunov exponents and Harder--Narasimhan filtration]
{Eigenvalues of Curvature, Lyapunov exponents and Harder-Narasimhan filtrations}
\author{Fei Yu}

\address{School of Mathematical Sciences, Zhejiang University, Hangzhou, 310027, People's Republic of China}
\email{ yufei@zju.edu.cn,vvyufei@gmail.com}
\address{School of Mathematical Sciences, Xiamen University, Xiamen, 361005, People's Republic of China}
\date{October 10, 2016
Supported by the Fundamental Research Funds for the Central Universities (No.
20720140526).}

\begin{abstract}
Inspired by Katz--Mazur theorem on crystalline cohomology and by
Eskin--Kontsevich--Zorich's numerical experiments, we conjecture
that the polygon of Lyapunov spectrum  lies above (or on) the
Harder--Narasimhan  polygon of the Hodge bundle over any
Teichm\"uller curve.  We also discuss the connections between the
two polygons and the integral of eigenvalues of the curvature of
the Hodge bundle by using Atiyah--Bott, Forni and M\"oller's works.
We obtain several applications to Teichm\"uller dynamics
conditional to the conjecture.
\end{abstract}
 \maketitle

 \tableofcontents

\section{Introduction}

\noindent Let $\mathcal{M}_g$ be the moduli space of Riemann surfaces
of genus $g$, and let $\mathcal{H}_g\rightarrow
\mathcal{M}_g$ be the bundle of pairs $(X,\omega)$,
where $\omega\neq 0$ is a holomorphic 1-form on $X\in \mathcal{M}_g$.
Denote by $\mathcal{H}_g(m_1,...,m_k)\hookrightarrow
\mathcal{H}_g$ the stratum of pairs $(X,\omega)$
for which the nonzero holomorphic $1$-form $\omega$
has $k$ distinct zeros of order $m_1,...,m_k$ respectively
(see~\cite{KZ03} for details).

There is a natural action of
$\operatorname{GL}_2^+(\mathbb{R})$ on
$\mathcal{H}_g(m_1,...,m_k)$, whose orbits project to complex
Teichm\"uller geodesics. The action of the
one-parameter diagonal subgroup of $\operatorname{SL}_2(\mathbb{R})$
defines the \textit{Teichm\"uller geodesic flow}; its orbit project
to geodesics in Teichm\"uller metric on $\mathcal{M}_g$. The
projection of the orbit of almost every point is
dense in the connected component of the ambient stratum.
Teichm\"uller geodesic flow has strong connections with flat
surfaces, billiards in polygons and interval exchange transformations
(see~\cite{Zo06} for a survey).

Fix an $\operatorname{SL}_2(\mathbb{R})$-invariant
finite ergodic measure $\mu$ on $\mathcal{H}_g$.
Zorich introduced the Lyapunov exponents for the Teichm\"uller
geodesic flow on $\mathcal{H}_g$
$$1=\lambda_1\geq \lambda_2\geq ...\geq\lambda_g\geq 0,$$
which measure the
logarithm of the growth rate of the Hodge norm of cohomology classes
under the parallel transport along the geodesic flow, see~\cite{Zo94}.

It is possible to evaluate Lyapunov exponents approximately through
computer simulation of the corresponding dynamical system. Such
experiments with Rauzy--Veech--Zorich induction (a discrete model of
the Teichm\"uller geodesic flow) performed in \cite{Zo96}, indicated
a surprising rationality of the sums $\lambda_1+...+\lambda_g$ of
Lyapunov exponents of the Hodge bundle with respect to the
Teichm\"uller geodesic flow on strata of Abelian and quadratic
differentials \cite{KZ97}.  An explanation of this phenomenon was
given by Kontsevich in \cite{Ko97} and then developed by Forni
\cite{Fo02}: this sum is, essentially, the
characteristic number of the determinant of the
Hodge bundle. Recently Eskin, Kontsevich and Zorich
have found the connection between the sum of
Lyapunov exponents and Siegel--Veech constants by establishing an
analytic Riemann--Roch formula, see~\cite{EKZ11}.

Zorich conjectured strict positivity of $\lambda_g$ and simplicity of
the spectrum of Lyapunov exponents for connected components of the
strata of Abelian differentials. Forni proved the first conjecture
in~\cite{Fo02}, Avila and Viana proved the second one in~\cite{AV07}.

We reproduce in the tables in Appendix the
approximate values of all individual Lyapunov exponents for connected
components of the strata of small genera using~\cite{KZ97}
and~\cite{EKZ11} as a source. Though the \textit{sum} of the top $g$
Lyapunov exponents is always rational for the strata, for the
Teichm\"uller curves, and, conjecturally, for all
$\operatorname{GL}(2,\mathbb R)$-invariant orbifolds, the individual
Lyapunov exponents seem to be completely transcendental and there are
no tools which would allow to evaluate them explicitly except several
very particular cases which we describe below.

Exact values of individual Lyapunov exponents can be
computed rigorously for certain invariant suborbifolds of the strata
of Abelian differentials. For example, Bainbridge~\cite{Ba07}
succeeded to perform such computation for suborbifolds in genus two.
(Since $\lambda_1$ is identically equal to $1$ for any
$\operatorname{GL}(2,\mathbb R)$-invariant orbifold, computation of $\lambda_2$
in genus $2$ is equivalent to the computation of the sum $\lambda_1+\lambda_2$.)

Such computation was also performed for certain special
Teichm\"uller curves. For the Teichm\"uller curves related to
triangle groups it was done by Bouw and M\"oller~\cite{BM10}, and by
Wright~\cite{Wr12b}; for square-tiled cyclic covers in~\cite{EKZ11}
and in~\cite{FMZ11a}; for square-tiled abelian covers by Wright
\cite{Wr12a}; for some wind-tree models by Delecroix, Hubert and
Leli\`evre~\cite{DHL11}.

Recall the definition of a Teichm\"uller curve in
$\mathcal{M}_g$. If the stabilizer
$\operatorname{SL}(X,\omega)\subset \operatorname{SL}_2(\mathbb{R})$
of a given pair $(X,\omega)$ forms a lattice, then
the projection of the orbit
$\operatorname{SL}_2(\mathbb{R})\cdot (X,\omega)$ to $\mathcal{M}_g$
gives a closed, algebraic curve called a
\textit{Teichm\"uller curve}. The relative canonical bundle over a
Teichm\"uller curve has a particularly simple and elegant
form~\eqref{canonical}; see Chen-M\"oller~\cite{CM11},
Eskin-Kontsevich-Zorich~\cite{EKZ11}. For any
Teichm\"uller curve, Kang Zuo and the author have introduced $g$
numbers:
$$1=w_1 \geq w_2\geq ...\geq w_g\geq 0\, ,$$
where $w_i$ is obtained by normalizing the slopes of
the Harder-Narasimhan filtration of the Hodge bundle. We can get
upper bounds of each $w_i$ by using some filtrations of the Hodge
bundle constructed using the special structure of
the relative canonical bundle formula \cite{YZ12a} \cite{YZ12b}.

Now we have a collection of numbers $\lambda_i$, where
$i=1,\dots,g$, measuring the stability of dynamical system and
a collection of numbers $w_i$, where $i=1,\dots,g$,
measuring the stability of algebraic geometry.
Tables in the Appendix provide
the numerical data for the numbers $\lambda_i$ corresponding to the
low genera strata and for the numbers $w_i$ corresponding to
Teichm\"uller curves in the corresponding strata.
It is natural to address a question, whether
there any relations between them?

Define the \textit{Lyapunov polygon} of the Hodge bundle over a
Teichm\"uller curve as the convex hull of the collection of points in
$\mathbb{R}^2$ having coordinates $(0,0)$, $(1,\lambda_1)$,
$(2,\lambda_1+\lambda_2), \dots, (g,\lambda_1+\dots+\lambda_g)$.
Similarly,
define the \textit{Harder--Narasimhan polygon} of the Hodge bundle over a
Teichm\"uller curve as the convex hull of the collection of points in
$\mathbb{R}^2$ having coordinates $(0,0)$, $(1,w_1)$,
$(2,w_1+w_2), \dots, (g,w_1+\dots+w_g)$.

Inspired by the Katz--Mazur theorem \cite{Ma72},
\cite{Ma73} which tells us that the Hodge polygon  lies above (or on)
the Newton polygon of the crystalline cohomology, we make the
following conjecture supported by all currently
available numerical data.

\begin{figure}[hb]
   %
   % Polygons $P_\lambda$ and $P_w$
   %
\includegraphics{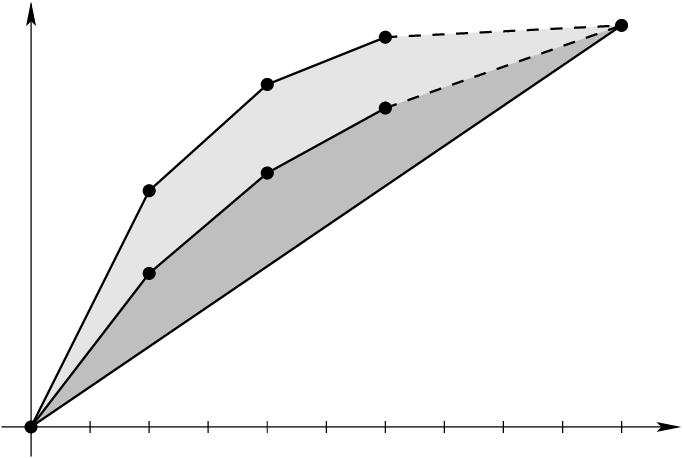}
\begin{picture}(0,0)(-43,-13)
\put(-50,-30){$P_\lambda$}
\put(-50,-50){$P_w$}
\put(-50,-50){$P_w$}
\put(-87.5,-102){\small $1$}
\put(-59,-102){\small $2$}
\put(-29.5,-102){\small $3$}
\put(-5,-102){\small $\dots$}
\put(27,-102){\small $g$}
\end{picture}
\vspace{100pt}
 \caption{
\label{fig:lambda:and:w:polygons}
We conjecture that the Lyapunov polygon $P_{\lambda}$ lies above
(or on) the Harder-Narasimhan  polygon  $P_{w}$.}
\end{figure}

\begin{mainconjecture}
For any Teichm\"uller curve,
the Lyapunov polygon of the Hodge
bundle lies above (or on) the
Harder--Narasimhan  polygon\footnote{By the time the manuscript
was submitted to the journal, a
proof of this conjecture was
announced by Eskin--Kontsevich--M\"oller--Zorich in~\cite{EKMZ}.}.
\end{mainconjecture}

Warning. Different articles have different definitions of ``lie
above'' and ``lie below'' for convex polygons. In the
context of the Conjecture above ``lies above'' is synonymous to
``contains as a subset'' since it is known that the two polygons share the rightmost and the
leftmost vertices, see Figure~\ref{fig:lambda:and:w:polygons}.
We discuss the notion ``lies above'' in a more general context in
Sections~\ref{s:slope:filtrations} and~\ref{s:convexity}.

In analytic terms, our Main Conjecture conjecture claims
that the following system of inequalities
is valid for any Teichm\"uller curve:
$$
\left\{
  \begin{aligned}[lr]
    \sum^i_{j=1} \lambda_j\geq \sum^i_{j=1} w_j & \text{ for } i=1,...,g-1; \\
    \sum^g_{j=1} \lambda_j= \sum^g_{j=1} w_j    &.
  \end{aligned}
\right. .
$$
where the equality for the last term $i=g$ is obtained by combining the Kontsevich formula for the sum of
the Lyapunov exponents of the Hodge bunde over a Teichm\"uller curve
(see Theorem~\ref{sumly} below) and the definition of the normalized
Harder--Narasimhan slopes $w_i$ (see section~\ref{ss:HNfiltrations};
see also~\cite{YZ12a} and~\cite{YZ12b}).

Equivalently, one can rewrite the latter system of inequalities as
$$
\overset{g}{\underset{j=i}{\sum}}\lambda_j\leq \overset{g}{\underset{j=i}{\sum}} w_j, \text{ for } i=2,...,g\,,
\text{ and }\overset{g}{\underset{j=1}{\sum}}\lambda_j=\overset{g}{\underset{j=1}{\sum}} w_j\,.
$$

The Main Conjecture is stated for the
Teichm\"uller curves. However, using the corollaries of recent
rigidity theorems of Eskin--Mirzakhani--Mohammadi the
statement of the Main Conjecture implies analogous estimates for other
$\operatorname{GL}(2,\mathbb{R})$-invariant suborbifolds, in
particular for the connected components of the strata. To illustrate
such applications we first recall the rigidity results.

\begin{NoNumberTheorem}{\cite[Theorem 2.3]{EMM13}}
\label{EMM}
Let $\mathcal{N}_n$ be a sequence of affine
$\operatorname{SL}(2,\mathbb{R})$-invariant manifolds, and suppose
$\nu_{\mathcal{N}_n}\rightarrow\nu$. Then $\nu$ is a probability
measure. Furthermore, $\nu$ is the affine
$\operatorname{SL}(2,\mathbb{R})-invariant$ measure
$\nu_{\mathcal{N}}$, where $\mathcal{N}$ is the smallest submanifold
with the following property: there exists some $n_0 \in\mathbb{N}$
such that $\mathcal{N}_n\subset \mathcal{N}$ for all $n>n_0$.
\end{NoNumberTheorem}

Bonatti, Eskin and Wilkinson use this theorem and a theorem of
Filip~\cite{Fi13a}, to give the following affirmative
answer to the question addressed by Matheus,
M\"oller and Yoccoz in~\cite{MMY13}.

\begin{theorem}[\cite{BEW14}]
\label{EBW}
Let $\mathcal{N}_n$ be a sequence of affine $\operatorname{SL}(2,\mathbb{R})$-invariant
manifolds, and suppose $\nu_{\mathcal{N}_n}\rightarrow\nu$. Then the
Lyapunov exponents of $\nu_{\mathcal{N}_n}$ converge to the Lyapunov
exponents of $\nu$.
\end{theorem}

As a corollary (conditional to the Main Conjecture) we prove the following conjecture
of Konsevich and Zorich~\cite{KZ97}.

\begin{corollary}
\label{cor:lambda:k:tends:to:1}
The Main Conjecture implies\footnote{As we already mentioned,
by the time the manuscript
was submitted to the journal, a
proof of this conjecture was
announced in~\cite{EKMZ}, so the Corollary becomes unconditional.}, in particular, that for any fixed positive
integer $k$ the Lyapunov exponent $\lambda_k$ of the Hodge bundle over
the hyperelliptic connected components $\mathcal{H}_g^{hyp}(2g-2)$ and
$\mathcal{H}_g^{hyp}(g-1,g-1)$ tends to $1$ as the genus tends to infinity:
$
\lambda_k \to 1\text{ as } g\to\infty \text{ for any fixed } k\in\mathbb{N}\,.
$
\end{corollary}

Note that for all other components of all other
strata of Abelian differentials the Lyapunov exponent conjecturally
tends to $\frac{1}{2}$ and not to $1$, see~\cite{KZ97}. It is a
challenging problem to deduce this asymptotics from the Main
Conjecture.

We prove Corollary~\ref{cor:lambda:k:tends:to:1},
obtain further results conditional to the Main Conjecture,
and state some further conjectures
in Section~\ref{s:Conditional:results}.

The Main Conjecture was first announced by the
author at the Oberwolfach conference \cite{Yu14}.
After that we realize that this conjecture is
analogous to the work of Atiyah--Bott on Hermitian Yang--Mills
metric~\cite{AB82}. This analogy was independently
noticed by M\"oller. Here we state the result of Atiyah--Bott in a
form for which the analogy is more transparent. Let $\varepsilon_j$, where
$1\leq j\leq g$, be the normalized integral of the $j$-th eigenvalue
of the curvature of the Hodge bundle
over a Teichm\"uller curve, see Forni~\cite{Fo02}
who proves the bounds
It follows that
$$1=\varepsilon_1\geq\varepsilon_2\geq...\geq\varepsilon_g\geq 0.$$
It follows from~\cite{AB82} that
$$ \left\{
  \begin{aligned}
    \sum^i_{j=1} \varepsilon_j\geq \sum^i_{j=1} w_j & \text{ for } i=1,...,g-1; \\
    \sum^g_{j=1} \varepsilon_j= \sum^g_{j=1} w_j    &.
  \end{aligned}
  \right.
$$
Thus, upper bounds for $w_i$ obtained
in~\cite{YZ12b} provide some information about $\varepsilon_i$.
Recall that partials sums of $\lambda_i$ and of $\varepsilon_i$  are
also related, see~\cite{Fo02} or Section~\ref{ss:varepsilon:geq:lambda}
for an outline of these results.

In Section~\ref{s:Teichmuller:curves} we review the definition of
Teichm\"uller curves, the formula for its relative
canonical bundle and the structure of natural
filtrations of the Hodge bundle over a Teichm\"uller
curve. Section~\ref{s:slope:filtrations} summarizes facts about
slope filtrations, especially Harder--Narasimhan filtrations.
It also recalls necessary facts about the integrals
of eigenvalues of the curvature. In Section~\ref{s:convexity} we
discuss various manifestations of convexity in
geometry and arithmetics. Section~\ref{ss:varepsilon:geq:w}
compares polygons of eigenvalue spectrum and
Harder--Narasimhan polygons. Section~\ref{ss:varepsilon:geq:lambda}
studies the relation between polygons of eigenvalue
spectrum and Lyapunov polygons. Finally,
Section~\ref{Hodge:and:Newton:polygons} compares
Hodge and Newton polygons. We start
Section~\ref{s:Conditional:results} with more detailed discussion of
the Main Conjecture. We proceed obtaining several applications
(conditional to the Main Conjecture) to Teichm\"uller dynamics. In
paricular, we present the proof of the old conjecture of
Kontsevich--Zorich stated in Corollary~\ref{cor:lambda:k:tends:to:1}.
We also prove a simple corollary $\lambda_i>0$ implies $w_i>0$ by
using Higgs bundles and we reprove
Eskin--Kontsevich--Zorich formula for the difference between sums of $\lambda^+$
and $\lambda^-$ Lyapunov exponents for Teichm\"uller curves and for connected
components. We provide certain numerical evidence for
the Main Conjecture in the Appendix.

\section{Teichm\"uller curves}
\label{s:Teichmuller:curves}

Teichm\"uller geodesic flow has close connections with flat surfaces,
billiards in polygons and interval exchange transformations;
see survey~\cite{Zo06} of Zorich covering many
important ideas of this field; see also survey~\cite{Mo12} by
M\"oller devoted to Teichm\"uller curves mainly from
the view point of algebraic geometry.

Denote by $\mathcal{H}_g(m_1,...,m_k)$ the stratum parameterized by
$(X,\omega)$ where $X$ is a curve of genus $g$ and $\omega$ is an
Abelian differential (i.e. a holomorphic one-form) on $X$ that has
$k$ distinct zeros of orders $m_1,...,m_k$. Let
$\overline{\mathcal{H}}_g(m_1,...,m_k)$ be the Deligne-Mumford
compactification of $\mathcal{H}_g(m_1,...,m_k)$.  Denote by
$\mathcal{H}^{hyp}_g(m_1,...,m_k)$ ( resp. odd, resp. even) the
hyperelliptic (resp.  odd theta characteristics,
resp. even theta characteristics) connected
component, see~\cite{KZ03}.

Let $\mathcal{Q}(d_1,...,d_n)$ be the stratum parameterizing $(Y,q)$
where $Y$ is a curve of genus $h$ and $q$ is a
meromorphic quadratic differentials with at most simple
poles on $Y$ that have $k$ distinct zeros of orders
$d_1,...,d_n$ respectively. If the quadratic differential is not a
global square of a $1$-form, there is a canonical double covering
$\pi\colon X\rightarrow Y$ such that $\pi^*q=\omega^2$,
where $\omega$ is already a holomorphic $1$-form.
This covering is ramified precisely at the zeros of odd order of $q$
and at the poles. It induces a map
\begin{equation}
\label{eq:Q:to:H}
\phi\colon\mathcal{Q}(d_1,...,d_n)\rightarrow\mathcal{H}_g(m_1,...,m_k)\,.
\end{equation}
A singularity of order $d_i$ of $q$ gives rise to two zeros of degree
$m=d_i/2$ when $d_i$ is even, and to a single zero
of degree $m=d_i+1$ when $d_i$ is odd. In
particular, any hyperelliptic locus in a stratum
$\mathcal{H}_g(m_1,...,m_k)$ is induced from a stratum
$\mathcal{Q}(d_1,...,d_n)$ satisfying $d_1+...+d_n=-4$, see~\cite{EKZ11}.

There is a natural action of $\operatorname{GL}_2^+(\mathbb{R})$ on
$\mathcal{H}_g(m_1,...,m_k)$, whose orbits project to complex
geodesics\footnote{Developing the results of
Eskin--Mirzakhani~\cite{EM13} and
Eskin--Mirzakhani--Mohammadi~\cite{EMM13}, Filip proved
in~\cite{Fi13a}, \cite{Fi13b}, that the closure of any such complex
geodesic is an algebraic variety.} with respect to
the Teichm\"uller metric on $\mathcal{M}_g$. The action
of the one-parrameter diagonal subgroup
$\begin{pmatrix}e^t & 0\\ 0 & e^{-t}\end{pmatrix}$, where $t\in\mathbb{R}$,
defines the \textit{Teichm\"uller geodesic flow}; its orbits
project to real Teichm\"uller geodesics in in $\mathcal{M}_g$.

It follows from the fundamental Theorems of Masur~\cite{Ms82} and
Veech~\cite{Ve82} that the $\operatorname{GL}_2^+(\mathbb{R})$-orbit of
almost any point $(X,\omega)$ in any stratum
$\mathcal{H}_g(m_1,...,m_k)$ of Abelian differentials is dense in the
ambient connected component of the stratum. The stabilizer
$\operatorname{SL}(X,\omega)\subset \operatorname{SL}_2(\mathbb{R})$
of almost any point $(X,\omega)$ is trivial.

The situation with some exceptional points $(X,\omega)$ is opposite:
the stabilizer $\operatorname{SL}(X,\omega)\subset
\operatorname{SL}_2(\mathbb{R})$ is as large as possible, namely it
forms a lattice in $\operatorname{SL}_2(\mathbb{R})$. By the results
of Smillie and Veech~\cite{Ve89} this happens if and only if the
$\operatorname{GL}^+_2(\mathbb{R})$-orbit of $(X,\omega)$ is closed
in the ambient stratum. The projection of such closed orbit to the
modulis space $\mathcal{M}_g$ gives a closed algebraic curve $C$ called
the \textit{Teichm\"{u}ller curve}
$$
\rho\colon C=\mathbb{H}/\operatorname{SL}(X,\omega) \rightarrow \mathcal{M}_g\,.
$$
By expression of McMullen~\cite{Mc06}, Teichm\"uller curves represent
\textit{closed complex geodesics}, meaning that every
Teichm\"uller curve is totally geodesic with respect to the
Teichm\"{u}ller metric on $\mathcal{M}_g$.

After suitable base change and compactification, we can get a
universal family $f\colon S\rightarrow C$, which is a relatively
minimal semi-stable model with disjoint sections $D_1,...,D_k$,
where the restrictions  $D_i|_X$
to each fiber $X$, is a zero of order $m_i$ of $\omega$,
see~\cite[p.11]{CM11}, \cite{Mo06}.

Let $\mathcal{L}\subset f_*{\omega_{S/C}}$ be the  line bundle
over the Teichm\"uller curve $C$ whose fiber over
the point corresponding to $(X,\omega)\in C$ is $\mathbb{C}\omega$,
the generating differential of the Teichm\"uller curve $C$.
This line bundle it known to be
``maximal Higgs'' (see~\cite{Mo06}). Let $\Delta\subset C$ be the set
of points with singular fibers. By definition of
``maximal Higgs'' bundle one has $\mathcal{L}\cong
\mathcal{L}^{-1}\otimes\omega_C(\mathrm{log}{\Delta})$,
see~\cite{VZ04}. The latter isomorphism implies the
following equality for the degree of $\mathcal{L}$:
$$
\chi:=2\mathrm{deg} \mathcal{L}=2g(C)-2+|\Delta|\,.
$$

The following particularly simple formula for the
relative canonical bundle can be found
in~\cite[p.18-19]{CM11}, or in~\cite[p.33]{EKZ11}):
\begin{align}\label{canonical}
 \omega_{S/C}\simeq f^*\mathcal{L}\otimes \mathcal{O}_S(\overset{k}{\underset{i=1}{\sum}}m_i D_i).
\end{align}
By the adjunction formula we get
$$D^2_i=-\omega_{S/C}D_i=-m_iD^2_i-\mathrm{deg}{\mathcal{L}},$$
and the self-intersection number of $D_i$ is
$$D^2_i=-\frac{1}{m_i+1}\frac{\chi}{2}\,.$$

Let $h^0(\mathcal{V})$ be the dimension of $H^0(X,\mathcal{V})$. If $0 \leq d_i\leq m_i$, then from
the exact sequence
$$0\rightarrow f_*\mathcal{O}(d_1D_1+...+d_kD_k)\rightarrow f_*\mathcal{O}(m_1D_1+...+m_kD_k)=f_*\omega_{S/C}\otimes\mathcal{L}^{-1}$$
and the fact that all sub-sheaves of a locally free sheaf on a curve
are locally free, we deduce that
$$f_*\mathcal{O}(d_1D_1+...+d_kD_k)\text{
is a vector subbundle of rank } h^0(d_1p_1+...+d_kp_k)\,,$$
where
$p_i$ is the intersection point of the section $ D_i$ and a generic fiber $F$.
Varying $d_i$ in the vector subbundles as above,
we have constructed in~\cite{YZ12a} numerous
filtrations of the Hodge bundle.

Examining the fundamental exact sequence
$$0\rightarrow f_*\mathcal{O}(\sum (d_i-a_i)D_i)\rightarrow
f_*\mathcal{O}(\sum d_iD_i)\rightarrow f_*\mathcal{O}_{\sum
a_iD_i}(\sum d_iD_i)\overset{\delta}{\rightarrow}$$

$$R^1f_*\mathcal{O}(\sum (d_i-a_i)D_i)\rightarrow
R^1f_*\mathcal{O}(\sum d_iD_i)\rightarrow 0$$
one can deduce certain nice properties of these filtrations. In particular, we have
\begin{lemma}[\cite{YZ12a}]\label{HN}
The Harder-Narasimhan filtration of $f_*\mathcal{O}_{aD}(dD)$ is
$$0\subset f_*\mathcal{O}_{D}((d-a+1)D)\subset ...\subset f_*\mathcal{O}_{(a-1)D}((d-1)D)\subset f_*\mathcal{O}_{aD}(dD)$$
and the direct sum of the graded quotient of this filtration is
$$\mathrm{grad}(HN(f_*\mathcal{O}_{aD}(dD)))=\overset{a-1}{\underset{i=0}{\oplus}}
\mathcal{O}_{D}((d-i)D).$$
\end{lemma}

By using those filtrations, we obtained Theorem~\ref{YZ12a} and
Theorem~\ref{YZ12b} below reproduced
from~\cite{YZ12a} and \cite{YZ12b}. They describe the
Harder--Narasimhan polygon of the Hodge bundle over a Teichm\"uller
curve.

\section{Slope filtrations}
\label{s:slope:filtrations}

Slope filtrations are present in algebraic and
analytic geometry, in asymptotic analysis, in ramification theory, in
$p$-adic theories, in geometry of numbers; see the
survey~\cite{An08} of Andr\'e. Five basic examples
include the Harder-Narasimhan filtration of
a holomorhic vector bundle over a smooth projective
curve, the Dieudonne--Manin filtration of $F$-isocrystals over a
$p$-adic point, the Turrittin--Levelt filtration of formal
differential modules, the Hasse--Arf filtration of finite Galois
representations of local fields, and the Grayson--Stuhler filtration
of Euclidean lattices. Despite the variety of their origins, these
filtrations share a lot of similar features.

Suppose that for some object $N$, there is a unique
descending \textit{slope filtration}
$$
0\subset F^{\geq \lambda_1}N\subset ... \subset F^{\geq \lambda_r}N= N
$$
for which $\lambda_1>...>\lambda_r$, such that
there is some natural way to associate the \textit{slope} $\lambda_i$
to every graded piece
$\mathrm{gr}^{\lambda_i}N=F^{\geq \lambda_i}N/F^{>\lambda_i}N$
(one says that the graded piece is \textit{
isoclinic of slope $\lambda_i$}). Denote
$\mathrm{rk}(\mathrm{gr}^{\lambda_i}N)$ by $n_i$, and let $n=\sum n_i$.
We shall call the sequence of pairs $(n_i,n_i\lambda_i), i=1,...,r$,
the \textit{type} of $N$. It is sometimes convenient to describe the type
equivalently by the single $n-$vector $\mu$ whose components are the
slopes $\lambda_i$ each represented $n_i$ times and arranged in
decreasing order. Thus
$$
\mu=(\mu_1,...,\mu_n)=
\big(\underbrace{\lambda_1,\dots, \lambda_1}_{n_1},
\underbrace{\lambda_2,\dots, \lambda_2}_{n_2},
\dots,
\underbrace{\lambda_r,\dots, \lambda_r}_{n_r}\big)
$$
with $\mu_1\geq\mu_2\geq...\geq\mu_n$, where the first $n_1$
entries are equal to $\lambda_1$, the next $n_2$
entries are equal to $\lambda_2$ and so on.

We introduce a partial ordering on the vectors $\mu$ that
parameterize our types.
This partial ordering is defined for vectors $\mu',\mu''$
having the same number $n$ of entries. We say that $\mu'\preceq\mu''$
when $\mu'_1+\dots+\mu'_i\le\mu''_1+\dots+\mu''_i$ for all $i=1,\dots,n$.
We also associate with every type $\mu$ a
convex polygon $P_{\mu}$ in the coordinate plane
$\mathbb{R}^2$ with vertices at the points having coordinates
\begin{equation}
\label{eq:vertices:of:P:mu}
(0,0),\,(1,\mu_1),\,(2,\mu_1+\mu_2),...,(n,\mu_1+...+\mu_n)\,.
\end{equation}
(see Figure~\ref{fig:Convex:polygon:P:mu}). It follows
from our definition of the partial ordering that
$\mu'\preceq\mu''$ if and only if for every pair of
vertices sharing the same first coordinate,
the vertex of $P_{\mu''}$ is located above the
corresponding vertex of $P_{\mu'}$ or coincides with it.

\begin{figure}[hbt]
   %
   % Polygon $P_\mu$
   %
\includegraphics{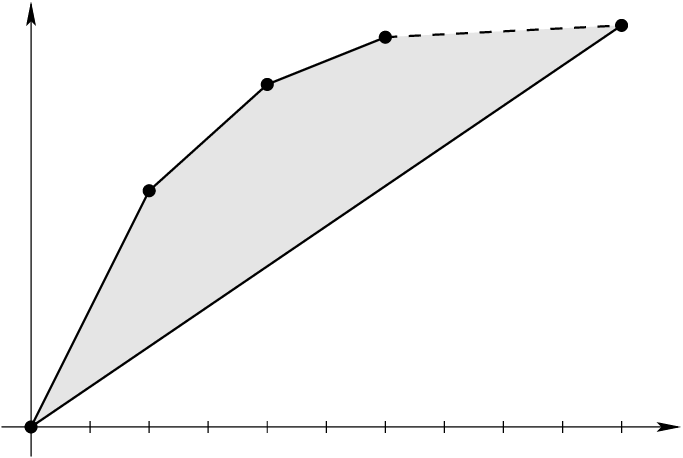}
\begin{picture}(0,0)(-43,-3)
\put(-50,-40){$P_\mu$}
\put(-87.5,-102){\small $1$}
\put(-110,-43){\small $(1,\mu_1)$}
\put(-59,-102){\small $2$}
\put(-103,-18){\small $(2,\mu_1+\mu_2)$}
\put(-29.5,-102){\small $3$}
\put(-35,-5){\small $\dots$}
\put(-5,-102){\small $\cdots$}
\put(27,-102){\small $n$}
\put(29,-4){\small $(n,\mu_1+\dots+\mu_n)$}
\end{picture}
\vspace{100pt}
 \caption{\label{fig:Convex:polygon:P:mu}Convex polygon $P_{\mu}$ .}
\end{figure}

Note that monotonicity of $\mu_i$ is equivalent
to convexity of the polygon $P_{\mu}$ with vertices at the collection of points~\eqref{eq:vertices:of:P:mu}.

%----------------------------------------------------------------------
\subsection{Eigenvalues of curvature: $\varepsilon$}
Forni introduced in~\cite{Fo02} the eigenvalues of curvature to study
Lyapunov exponents of the Hodge bundle. Here we
follow~\cite{FMZ11} whose setup is closer to the
current paper.

Let
$$f\colon \overline{\mathcal{M}}_{g,1}\rightarrow \overline{\mathcal{M}}_g$$
be the natural forgetful map from the
compactified
moduli space $\mathcal{M}_{g,1}$
of pairs
$(X,p)$,
where $p\in X$,
to the compactified moduli space
$\mathcal{M}_g$ of Riemann surfaces
$X$ of genus $g$.

For the weight one $\mathbb{Q}$-VHS
$$(R^1f_*\mathbb{Q},H^{1,0}=f_*\omega_{\overline{\mathcal{M}}_{g,1}/
\overline{\mathcal{M}}_g}\subset
H=(R^1f_*\mathbb{Q}\otimes_{\mathbb{Q}}\mathcal{O}_{\mathcal{M}_g})_{ext})$$
the flat Gauss--Manin connection $\bigtriangledown$ composed with the
inclusion and projection gives a map
$$A^{1,0}\colon H^{1,0}\rightarrow H \rightarrow  H\otimes \Omega_{\overline{\mathcal{M}}_g}(\mathrm{log} (\overline{\mathcal{M}}_g\backslash \mathcal{M}_g)) \rightarrow (H/H^{1,0})\otimes\Omega_{\overline{\mathcal{M}}_g}(\mathrm{log} (\overline{\mathcal{M}}_g\backslash \mathcal{M}_g))\,,$$
which is $\mathcal{O}_{\overline{\mathcal{M}}_g}$-linear.

The map $A^{1,0}$ is the second fundamental form
of the Hodge bundle,
which is also know as the Kodaira--Spencer map.
Being restricted to a
curve $C$ in $\overline{\mathcal{M}}_g$,
$A^{1,0}\wedge A^{1,0}=0$.
This condition (which is, actually, void for curves)
defines a \textit{Higgs field} which is discussed in Section~\ref{ss:Higgs:fields}.

Denote by $\Theta_H,\Theta_{H^{1,0}},\Theta_{H^{0,1}}$
the curvature tensor of the metric connections
of the holomorphic Hermitian bundles $H,H^{1,0},H^{0,1}$.
By Cartan's structure equation,
$$\Theta_H=  \begin{bmatrix}
 \Theta_{H^{1,0}}- \overline{A^{1,0}}^T\wedge A^{1,0}& *  \\
  *   & \Theta_{H^{0,1}}-  A^{1,0}\wedge  \overline{A^{1,0}}^T
  \end{bmatrix}$$
It follows that
$$\Theta_{H^{1,0}}=\Theta_H|_{H^{1,0}}+ \overline{A^{1,0}}^T\wedge A^{1,0}.$$

Note that $\Theta_H$ is the curvature of the Gauss--Manin connection,
which is flat. So $\Theta_H$ is null, and the curvature
$\Theta_{H^{1,0}}$ of the Hodge bundle
can be expressed as:
$$\Theta_{H^{1,0}}=\overline{A^{1,0}}^T\wedge A^{1,0}\,.$$

We work with the pullbacks of the vector bundles $H,H^{1,0},H^{0,1}$
to the moduli spaces $\mathcal{H}_g$ or $\mathcal{Q}$ of Abelian
(correspondingly quadratic) differentials with respect to the natural
projections $\rho\colon \mathcal{H}_g\rightarrow
\mathcal{M}_g$(correspondingly $\varrho\colon
\mathcal{Q}\rightarrow \mathcal{M}_g$). For any pair
$(X,q)$
we can view the holomorphic quadratic differential $q$
as the tangent vector $v=q$ to the moduli space $\mathcal{M}_g$ at the
point $X$
under the identification between the bundle of holomorhic
quadratic differentials and the tangent bundle of the moduli space of
Riemann surfaces through Beltrami differentials.
We can plug the vector $v$ into the
1-form $A^{1,0}$ with values in linear maps to define a linear map
$$A_{q}:H^{1,0}(X)\rightarrow H^{0,1}(X).$$
for every point $(X,q)$ of the moduli space $\mathcal{Q}$,
see~\cite[p.8]{FMZ11} for details.Analogously,
for any Abelian differetianl $\omega$, let $A_{\omega}:=A_q$ be the complex-linear map corresponding to the quadratic differential $q=\omega^2$.

Following Forni,
for any $\alpha,\beta\in H^{1,0}(X)$, define:
$$B_{\omega}(\alpha,\beta):=\frac{i}{2}\int_X\frac{\alpha\beta}{\omega}\overline{\omega}\,.$$
The complex-valued symmetric bilinear form
$B_{\omega}$ depends continuously
(actually, even real-analytically) on the Abelian
differential $\omega$. The second fundamental form $A_{\omega}$ can
be expressed
in terms of the complex-valued symmetric bilinear form
$B_{\omega}$ in the following way, see~\cite{Fo02},
\cite[Lemma 2.1]{FMZ11}:
$$(A_{\omega}(\alpha),\overline{\beta})=-B_{\omega}(\alpha,\beta).$$
It is related to the derivative
of the period matrix along the Teichm\"uller geodesic flow.

For any Abelian differential $\omega$, let $H_{\omega}$ be the
\textit{negative} of the Hermitian curvature form $\Theta_{\omega}$
on $H^{1,0}(X)$. Let $B$ be the matrix of the bilinear form
$B_{\omega}$ on $H^{1,0}(X)$ with respect to some orthonormal basis
$\Omega:=\{\omega_1,...,\omega_g\}$
of holomorphic Abelian differentials $\omega_1,\dots,\omega_g$
on $X$, that is:
$$B_{jk}:=\frac{i}{2}\int_{X}\frac{\omega_j\omega_k}{\omega}\overline{\omega}.$$
The Hermitian form $H_{\omega}$ is positive-semidefinite and its matrix $H$ with respect to any
Hodge-orthonormal basis $\Omega$ can be written as follows \cite{Fo02}\cite{FMZ11}:
$$H=B\cdot \overline{B}^T.$$

Let $EV(H_{\omega})$ and $EV(B_{\omega})$ denote the set of eigenvalues of the forms $H_{\omega}$ and $B_{\omega}$ respectively. The following identity holds:
$$EV(H_{\omega})=\{|\lambda|^2\text{ where }\lambda\in EV(B_{\omega})\}.$$
For every Ableian differential $\omega$, the eigenvalues of the
positive semidefinite from $H_{\omega}$ on $H^{1,0}(X)$ will be
denoted as follows:
$$1=\Lambda_1(\omega)>\Lambda_2(\omega)\geq...\geq\Lambda_g(\omega)\geq 0\,,$$
where the identity $\Lambda_1(\omega)=1$ is proved
in~\cite{Fo02}, \cite[p.16]{FMZ11}. Every eigenvalue
as above gives a well-defined continuous, non-negative, bounded
function on the moduli space of all (normalized) abelian
differentials.

For a Teichm\"uller curve $C$, there is a
natural volume form $\mathrm{d}\sigma$ which satisfies
$$\frac{i}{2\pi}\Theta_{H^{1,0}}=H_{\omega}\mathrm{d}\sigma\,.$$
This volume form
coincides with the normalized hyperbolic area form
\begin{equation}
\label{eq:d:sigma}
\mathrm{d}\sigma=\frac{1}{\pi}\,\mathrm{d}g_{hyp}(\omega)
\end{equation}
associated to the canonical hyperbolic metric
of constant negative curvature $-4$
on the Teichm\"uller curve $C$ used in~\cite[p.32]{EKZ11}.
Thus, we have
$$\int_C\Lambda_1(\omega)\mathrm{d}\sigma=\int_C\mathrm{d}\sigma=\frac{\chi}{2}\,,$$
where $-\chi$ is the Euler characteristic of
the Teichm\"uller curve $C$ with punctures at the cusps,
$\chi=2g-2+|\Delta|$, and $|\Delta|$ is the number of cusps of $C$.

Following Forni \cite{Fo02}, we define the integrals
\begin{align}\label{ki}
 \varepsilon_j=\frac{1}{\chi/2}\int_C\Lambda_j(\omega)\,\mathrm{d}\sigma\,.
\end{align}
By definition, the numbers $\epsilon_1, \dots,\epsilon_g$ satisfy inequalities:
$1=\varepsilon_1\geq ....\geq\varepsilon_g\geq 0$.
We define the \textit{eigenvalue type} $\varepsilon(C)$
of a Teichm\"uller curve $C$ as
$$\varepsilon(C)=(\varepsilon_1,...,\varepsilon_g)\,.$$

%---------------------------------------------
\subsection{Lyapunov exponents: $\lambda$}
Zorich introduces the Lyapunov exponents of the Hodge
bundle to study the Teichm\"uller geodesic flow~\cite{Zo94}. The
geometric meaning of these Lyapunov
exponents is clearly explained in \cite[section 4]{Zo06}.

A motivating example called Ehrenfest wind-tree model for Lorenz
gases appears in the work of Delecroix, Hubert and Leli\'evre
\cite{DHL11}. Consider a billiard on the plane with
$\mathbb{Z}^2$-periodic rectangular obstacles as in Figure~\ref{fig:windtree}.

\begin{figure}[htb]
\label{wt}
\centering
\setlength{\unitlength}{0.5mm}
Ehrenfest wind-tree model for Lorenz gases \cite{DHL11}.\\
\begin{picture}(100,100)
\newsavebox{\xbox}
\savebox{\xbox}
  (20,16)[bl]{\put(0,0){\line(1,0){20}
 \put(0,0){\line(0,1){16}}
  %\put(10,8){\circle*{2}}
 \put(20,0){\line(0,1){16}}
 \put(0,16){\line(1,0){20}}
}}

\multiput(10,12)(30,0){3}{\usebox{\xbox}}
\multiput(10,42)(30,0){3}{\usebox{\xbox}}
\multiput(10,72)(30,0){3}{\usebox{\xbox}}

\put(53,28){\line(-1,1){23}}
\put(30,51){\line(1,1){21}}
\put(51,72){\line(1,-1){19}}
\put(70,53){\line(-1,-1){10}}
\put(60,43){\line(1,-1){15}}
\put(75,28){\line(1,1){14}}
\put(89,42){\vector(1,-1){11}}

\end{picture}
% $$\lambda_2=\underset{time\rightarrow \infty}{\mathrm{lim}}\mathrm{sup}\frac{\mathrm{log}(\textit{rate of billiard trajectory escapes to infinity})}{\mathrm{log} (\textit{time})}=\frac{2}{3}.$$
\caption{
\label{fig:windtree}
Billiard in the plane with periodic rectangular obstacles.}
\end{figure}

It is shown in~\cite{DHL11} that for
all parameters $(a,b)$ of the obstacle (i.e.,
for all pairs of lenghts $a, b\in(0, 1)$ of the
sides of the rectangular obstacles), for almost all initial direction
$\theta$, and for any starting point $x$ the diameter
of the billiard trajectory grows with the rate $t^{2/3}$:
$$\lambda_2=
\limsup_{t\to\infty}\frac{\log\left(\textit{distance between }x\textit{ and }\phi^{\theta}_{t}(x)\right)}{\log t}=
\frac{2}{3}\,.$$

The number ``$\frac{2}{3}$'' here is the Lyapunov exponent of a certain renormalizing
dynamical system associated to the initial one.

We recall now the definition of Lyapunov exponents
of the Hodge bundle. Fix an $\operatorname{SL}_2(\mathbb{R})$-invariant, ergodic measure $\mu$ on
$\mathcal{H}_g$. Let $V$ be the restriction of the real Hodge
bundle (i.e. the bundle with fibers $H^1(X,\mathbb{R})$) to the
support $\mathcal{M}\subset\mathcal{H}_g$ of $\mu$. Let $S_t$ be the lift of the geodesic flow to
$V$ via the Gauss--Manin connection. Then\textit{ Oseledec's multiplicative
ergodic Theorem} guarantees the existence of a filtration
$$0\subset V_{\lambda_g}\subset ...\subset V_{\lambda_1}=V$$
by measurable vector subbundles with the property that, for almost
all $m\in \mathcal{M}$ and all $v\in V_m\backslash\{0\}$ one has
$$||S_t(v)||=\mathrm{exp}(\lambda_it+o(t)),$$
where $i$ is the maximal index such that $v$ is in the fiber of
$V_i$ over $m$ (i.e. $v\in(V_i)_m$). The numbers $\lambda_i$ for
$i=1,...,k\leq \mathrm{rank}(V)$ are called the \textit{Lyapunov exponents}
of the \textit{ Kontsevoch-Zorich cocycle} $S_t$. Since $V$ is symplectic, the spectrum
of Lyapunov exponents is symmetric in the
sense that $\lambda_{g+k}=-\lambda_{g-k+1}$. Moreover, from
elementary geometric arguments it follows that one always has
$\lambda_1=1$. Thus, the Lyapunov spectrum is
completely determined by the non-negative Lyapunov exponents
$$1=\lambda_1\geq\lambda_2\geq...\geq\lambda_g\geq 0.$$

We will apply Oseledec¡¯s theorem in two instances. The first
one corresponds to the
Masur--Veech measures $\mu_{gen}$.
The support of such measure coincdes with the
entire
hypersurface of flat surfaces of area one in a connected component
of a stratum of Abelian or quadratic differentials.
The second case corresponds to Teichm\"uller curves.
When talking about Lyapunov exponents for Teichm\"uller curves $C$ we
take  $\mu$ to be the measure on the unit tangent bundle $T^1C$ to a
Teichm\"uller curve that stems from the Poincar\'{e} metric $g_{hyp}$
on $\mathbb{H}$ with scalar curvature $-4$. In both cases, the
integrability condition of Oseledets theorem is known to
be satisfied, see e.g.~\cite[p.38]{Mo12}.

We define the \textit{Lyapunov type} $\lambda(C)$
of a Teichm\"uller curve $C$ as
$$\lambda(C)=(\lambda_1,...,\lambda_g)\,.$$

A bridge between the ``dynamical'' definition of Lyapunov exponents
and the  ``algebraic'' method applied in the sequel
originates from
the
following result. It is first formulated by Kontsevich~\cite{Ko97}(in a slightly different
form) and then extended by Forni~\cite{Fo02}.

\begin{theorem}[\cite{Ko97}, \cite{Fo02}, \cite{BM10}]
\label{sumly}
If the VHS over the Teichm\"uller curve $C$ contains a sub-VHS $\mathbb{W}$
of rank $2k$, then the sum of the $k$ corresponding non-negative
Lyapunov exponents equals
$$\overset{k}{\underset{i=1}{\sum}}\lambda^{\mathbb{W}}_i=\frac{2\mathrm{deg} \mathbb{W}^{(1,0)}}{2g(C)-2+|\Delta|},$$
where $\mathbb{W}^{(1,0)}$ is the $(1,0)$-part of the Hodge
filtration of the vector bundle associated with $\mathbb{W}$
and $|\Delta|$ is the number of cusps of $C$. In
particular, we have
$$\overset{g}{\underset{i=1}{\sum}}\lambda_i=\overset{g}{\underset{i=1}{\sum}}\varepsilon_i=\frac{2\mathrm{deg}f_*\omega_{S/C}}{2g(C)-2+|\Delta|}.$$
\end{theorem}
The formula immediately
implies the Arakelov inequality for
Teichm\"uller curves:
$$\mathrm{deg}f_*\omega_{S/C}=
\left(\frac{1}{2}\ \overset{g}{\underset{i=1}{\sum}}\lambda_i\right)
\Big(2g(C)-2+|\Delta|\Big)\leq
\frac{g}{2}\Big(2g(C)-2+|\Delta|\Big)\,.$$

Eskin, Kontsevich and Zorich have elaborated
an appropriate analytic Riemann-Roch formula to compute the sum of Lyapunov exponents
of the Hodge bundle along the
Teichm\"{u}ller geodesic flow on any $\operatorname{SL}(2,\mathbb{R})$-invariant suborbifold.

\begin{theorem}[{\cite[Theorem 1]{EKZ11}}]
Let $\mathcal{M}_1$ be any closed connected
$\operatorname{SL}(2,\mathbb{R})$-invariant suborbifold of some stratum
$\mathcal{M}_g(m_1,...m_n)$ of Abelian differentials, where
$m_1+...+m_n=2g-2$. The top $g$ Lyapunov exponents of the Hodge bundle
over $\mathcal{M}_1$ along the Teichm\"uller flow satisfy the
following relation:
$$\overset{g}{\underset{i=1}{\sum}}\lambda_i=\frac{1}{12}\ \overset{k}{\underset{i=1}{\sum}}\frac{m_i(m_i+2)}{m_i+1}+\frac{\pi^2}{3}\,c_{area}(\mathcal{M}_1),$$
where $c_{area}(\mathcal{M}_1)$ is the area Siegel-Veech constant
corresponding to the suborbifold $\mathcal{M}_1$. The leading Lyapunov exponent $\lambda_1$ is equal to one.
\end{theorem}

%----------------------------------------------------------------------
\subsection{Harder-Narasimhan filtrations: $w$}
\label{ss:HNfiltrations}

We refer the readers to~\cite{HN75}, and to~\cite{HL97}
for details about the Harder-Narasimhan filtration.

Consider a smooth curve $C$ and a holomorphic vector
bundle $V$ over $C$. In order to recall the
definition of stability it would be consentient to
introduce  the normalized Chern class or \textit{slope}
$\mu(V)=\mathrm{deg}(V)/\mathrm{rk}(V)$ of the vector
bundle $V$. A holomorphic bundle $V$ is called
\textit{stable} if for every proper holomorphic subbundle $W$ of $V$,
we have $\mu(W)<\mu(V)$. A \textit{semi-stable} bundle is defined
similarly but we allow now the weak inequality $\mu(W)\leq\mu(V)$.

Harder and Narasimhan show that every holomorphic bundle $V$ has a canonical filtration
$$0=HN_0(V)\subset HN_1(V)\subset ...\subset HN_r(V)=V$$
satisfying the following two properties. Every
graded quotient
$$
\mathrm{gr}^{HN}_i=HN_i(V)/HN_{i-1}(V)\,,\ \text{ for }i=1,\dots,r,
$$
is semi-stable and
$$\mu(\mathrm{gr}^{HN}_1)>\mu(\mathrm{gr}^{HN}_2)>...>\mu(\mathrm{gr}^{HN}_r).$$
If $\mathrm{gr}^{HN}_i$ has rank $n_i$ and Chern
number $k_i$, so that $n=\sum n_i,k=\sum k_i$, we
shall call the sequence of pairs $(n_i,k_i),\ i=1,...,r$ the
\textit{slope type} of
the holomorphic bundle
$V$. As before, it is convenient to describe the type equivalently by
a single $n$-vector $\mu(V)$ whose components are the ratios
$k_i/n_i$ each represented $n_i$ times and arranged in decreasing
order
$$\mu(V)=(\mu_1,...,\mu_n)
=\Big(
\underbrace{k_1/n_1,\dots,k_1/n_1}_{n_1},\,\dots\,,
\underbrace{k_r/n_r,\dots,k_r/n_r}_{n_r}
\Big)\,.
$$

For a Teichm\"uller curve $C$, it is convenient to set
$w_i=\mu_i(f_*\omega_{S/C})/(\chi/2)$ and
to define the \textit{Harder--Narasimhan type} of
a Teichm\"uller curve $C$ as
the $g$-vector
\begin{align}\label{wi}
w(C)=(w_1,...,w_g).
\end{align}

It follows from geometric arguments in~\cite{Mo06},
see also~\cite{Gj12} and~\cite{Wr12b}, that for any \mbox{Teichm\"uller} curve $C$ in any
stratum of Abelian differentials the identity $w_1(C)=1$ is valid.

The Harder--Narasimhan type of a Teichm\"uller curve is given by the
following two theorems:

\begin{theorem}[\cite{YZ12a}]\label{YZ12a}Let $C$ be a Teichm\"{u}ller curve in the hyperelliptic locus of some
stratum $\overline{\mathcal{H}}_g(m_1,...,m_k)$, and denote by
$(d_1,...,d_n)$ the orders of singularities of underlying quadratic
differentials. Then $w_i(C)$ is the $i$-th largest number in
the following set
$$\{1\}\cup\Big\{1-\frac{2k}{d_j+2}\Big\}_{\forall d_j, 0<2k\leq d_j+1}$$
In particular, the Harder--Narasimhan type $w(C)$ of a Teichm\"uller curve
in any hyperelliptic locus of any stratum is constant and depends only on the locus.
\end{theorem}

The nonvarying property of the
the Harder--Narasimhan type of all Teichm\"uller curves is also
valid for certain strata in low genera $g=3,4,5$, see tables
with explicit values of all $w_i(C)$ in the Appendix.
This observation
provides an alternative proof~\cite{YZ12a} of the Kontsevich--Zorich conjecture
on non-varying of the sums of the Lyapunov exponents of the Hodge bundle for
the corresponding strata; see~\cite{CM11} for the original proof.

Since the
action of $\operatorname{GL}^+(2,\mathbb{R})$
preserves the strata of meromorphic quadratic differentials
with at most simple poles, it also preserves their images under the map~\eqref{eq:Q:to:H}.
In particular, all hyperelliptic loci in the strata of Abelian differentials are
invariant under the action of $\operatorname{GL}^+(2,\mathbb{R})$. Thus, if some
Teichm\"uller curve $C$ intersects some hyperelliptic locus $\mathcal{M}^{hyp}$, it is
entirely contained in it, $C\subset\mathcal{M}^{hyp}$.

\begin{theorem}[\cite{YZ12b}]\label{YZ12b}
\label{th:inequalities:on:w}
For any stratum $\mathcal{H}_g(m_1,...,m_k)$
of Abelian differentials, order the entries of the set with multiplicites
$$\left\{\cfrac{j}{m_l+1}\right\}_{\substack{1\leq j\leq m_l\\1\leq l\leq k}}$$
getting an increasing sequence of $2g-2=$
numbers $a_1\leq a_2\leq \cdot\cdot\cdot \leq
a_{2g-2}$, where $2g-2=m_1+\dots+m_k$.

For any Teichmu\"ller curve $C$ in the stratum
$\mathcal{H}_g(m_1,...,m_k)$, there exists a
permutation $\myperm_C$ of the set $\{1,\dots,2g-2\}$
satisfying the following properties. For $i = 2, . . . , g$, $\myperm_C(i)\ge 2i-2$,
where all inequalities for $i = 2, . . . , g-1$ are strict if $C$ is not contained
in some hyperelliptic locus. The normalized Harder--Narasimhan
slopes $w_i(C)$ of the Hodge bundle over $C$ satisfy the following system of
inequalities:
\begin{equation}
\label{eq:bounds:for:w}
w_i\leq 1-a_{\myperm_C(i)}\ \text{ for }i=2,\dots, g\,,
\end{equation}
\end{theorem}

\begin{example}Consided the stratum $\mathcal{H}_5(6,1,1)$.
Ordering the entries of the set with multiplicities
$$\left\{\frac{1}{7},\frac{2}{7},\frac{3}{7},\frac{4}{7},\frac{5}{7},\frac{6}{7},\frac{1}{2},\frac{1}{2}\right\}$$
in increasing order we get an order set with multiplicities
$$
\{a_1, a_2, \dots, a_8\}=
\left\{\frac{1}{7},\frac{2}{7},\frac{3}{7},\frac{1}{2},\frac{1}{2},\frac{4}{7},\frac{5}{7},\frac{6}{7}\right\}\,.$$
For any a Teichm\"{u}ller
curve in the stratum $\mathcal{H}_5(6,1,1)$, the permutation $\myperm_C$ of the set $\{1,\dots, 8\}$ satisfies
$\myperm_C(i)\ge 2i-2$ for $i=1,\dots,4$, so we have
$$\myperm_C(2)\geq 2,\quad \myperm_C(3)\geq 4,\quad \myperm_C(4)\geq 6,\quad \myperm_C(5)=8\,.$$
Theorem~\ref{th:inequalities:on:w} asserts that
 the normalized Harder--Narasimhan slopes $w_i(C)$
satisfy the following inequalities:
\begin{align*}
w_2(C)&\leq 1-a_{\myperm_C(2)}\leq 1-a_2=\frac{5}{7}\\
w_3(C)&\leq 1-a_{\myperm_C(3)}\leq 1-a_4=\frac{1}{2}\\
w_4(C)&\leq 1-a_{\myperm_C(4)}\leq 1-a_6=\frac{3}{7}\\
w_5(C)&\leq 1-a_{\myperm_C(5)}\leq 1-a_8=\frac{1}{7}\,.
\end{align*}

If $C$ is not located in some hyperelliptic locus, then
$$\myperm_C(2)\geq 3,\quad \myperm_C(3)\geq 5,\quad \myperm_C(4)\geq 7,\quad \myperm_C(5)=8\,.$$
Theorem~\ref{th:inequalities:on:w} asserts that
 the normalized Harder--Narasimhan slopes $w_i(C)$
satisfy the following inequalities:
\begin{align*}
w_2(C)&\leq 1-a_{\myperm_C(2)}\leq 1-a_3=\frac{4}{7}\\
w_3(C)&\leq 1-a_{\myperm_C(3)}\leq 1-a_5=\frac{1}{2}\\
w_4(C)&\leq 1-a_{\myperm_C(4)}\leq 1-a_7=\frac{2}{7}\\
w_5(C)&\leq 1-a_{\myperm_C(5)}\leq 1-a_8=\frac{1}{7}\,.
\end{align*}

   %.
\end{example}

A simple corollary of this theorem is
\begin{corollary}[\cite{YZ12b}]For a Teichm\"{u}ller
curve which lies in $\mathcal{H}_g(m_1,...m_k)$, we have
inequalities:
$$\overset{g}{\underset{i=1}{\sum}}\lambda_i=\overset{g}{\underset{i=1}{\sum}}\varepsilon_i=\overset{g}{\underset{i=1}{\sum}}w_i\leq \frac{g+1}{2}.$$
\end{corollary}
The two equalities in the above formula are direct corollaries
of Theorem~\ref{sumly}.

%---------------------------------------------------------------------
\section{Convexity}
\label{s:convexity}

In \cite[section 12]{AB82}, Atiyah and Bott discussed the convexity
of polygons and the relation with Hermitian matrices. Shatz defines
the partial ordering by
$$\lambda\succeq \mu \textit{ if }P_{\lambda}\textit{ is above }P_{\mu}.$$
If we consider $P_{\mu}$ as the graph of a concave function $p_{\mu}$, then $p_{\mu}$ is defined on the integers by
$$p_{\mu}(i)=\sum_{j\leq i}\mu_j$$
and interpolates linearly between integers. Here the $\mu_j$ are the components of our $n-$vector $\mu$.

\begin{figure}[hb]
   %
   % Polygons $P_\lambda$ and $P_{\mu}$
   %
\includegraphics{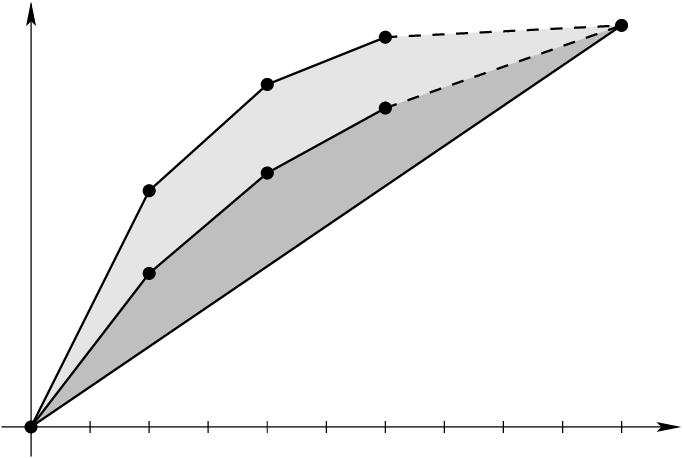}
\begin{picture}(0,0)(-43,-13)
\put(-50,-30){$P_\lambda$}
\put(-50,-50){$P_{\mu}$}
\put(-87.5,-102){\small $1$}
\put(-59,-102){\small $2$}
\put(-29.5,-102){\small $3$}
\put(-5,-102){\small $\dots$}
\put(27,-102){\small $n$}
\end{picture}
\vspace{100pt}
\caption{The $P_{\lambda}$ lies above (or on) the $P_{\mu}$.}
\end{figure}

Hence, for our vector notation, it translates in to the following partial ordering:
 $$\lambda\succeq \mu \Leftrightarrow \left\{
  \begin{aligned}
    \sum^i_{j=1} \lambda_j\geq \sum^i_{j=1} \mu_j & \textit{ for } i=1,...,n-1; \\
    \sum^n_{j=1} \lambda_j= \sum^n_{j=1} \mu_j    &
  \end{aligned}
  \right.
.$$
This partial ordering on vectors in $\mathbb{R}^n$ is well known in various contexts.

This partial ordering occurs in Horn \cite{Ho54} where it is shown to be equivalent to either of the following properties
\begin{align}\label{confun}
\sum_j f(\mu_i)\leq \sum_j f(\lambda_j) \textit{ for every convex function } f:\mathbb{R}\rightarrow \mathbb{R};
\end{align}
$$ \mu=P\lambda \textit{ where } \lambda,\mu\in \mathbb{R}^n \text{ and } P \textit{ is a doubly stochastic matrix.}$$
We recall that a real square matrix is \textit{stochastic} if $p_{ij}\geq0$ and $\underset{j}{\sum} p_{ij}=1$ for all $i$. If in addition the transposed matrix is also stochastic then $P$ is called \textit{doubly stochastic}. A theorem of Birkhoff identifies doubly stochastic matrices in terms of permutation matrices, namely
$$\textit{The doubly stochastic } n\times n \textit{ matrices are the convex hull of the permutation matrices}.$$
Now the equivalence relation can be replace by
$$\widehat{\Sigma_n\mu}\subseteq \widehat{\Sigma_n\lambda}$$
where $\Sigma_nx$ denotes the orbit of any $x\in \mathbb{R}^n$
under the permutation group $\Sigma_n$ , and $\hat{C}$
denotes the convex hull of the set $C\subset \mathbb{R}^n$.

Schur showed that if $\mu_j (j = 1, ..., n)$ are the diagonal elements of a Hermitian matrix whose
eigenvalues are $\lambda_j$, then $\mu\preceq\lambda$. We give the proof for the largest eigenvalue, the proof is similar for general cases.

\begin{lemma}[Schur]\label{schur}
For a Hermitian matrix $H=[h_{ij}]$, let $\lambda_1\geq...\geq\lambda_n$ be its eigenvalues, then
$$\lambda_1\geq h_{ii}.$$
\end{lemma}
\begin{proof}
Let $U=[u_{ij}]$ be a unitary matrix such that
$$H=U \mathrm{diag}[\lambda_1,...,\lambda_n]\overline{U}^T.$$
Since $\underset{j}{\sum}u_{ij}\overline{u}_{ij}=1$, we have
$$\lambda_1=\lambda_1(\underset{j}{\sum}u_{ij}\overline{u}_{ij})\geq \underset{j}{\sum} \lambda_ju_{ij}\overline{u}_{ij}=h_{ii}.$$
\end{proof}

Horn proved the converse so that another necessary and equivalent condition of the relation $\mu\preceq\lambda$.

\begin{NoNumberProposition}
Vectors $\mu$ and $\lambda$ satisfy relation
$\mu\preceq\lambda$ if and only if there exists a
Hermitian matrix with diagonal elements $\mu_j$ and eigenvalues
$\lambda_j$.
\end{NoNumberProposition}

For a general compact Lie group $G$, the role of the Hermitian (or rather skew-Hermitian)
matrices is played now by the Lie algebra $\mathfrak{g}$ of $G$. The diagonal matrices are replaced by the Lie
algebra $\mathfrak{t}$ of a maximal torus $T$ of $G$ and $\Sigma_n$ becomes the Weyl group $W$. Writing a set of  $\lambda_j$ in decreasing order corresponds to picking a (closed) positive Weyl chamber $C$ in $\mathfrak{t}$: this is a fundamental
domain for the action of $W$. (\cite{AB82})

We also need the following linear algebra fact:
\begin{lemma}\label{hb}
For a complex symmetric matrix $B=[b_{ij}]$, $H=B\overline{B}^T=[h_{ij}]$, $\alpha=[a_1,...,a_n]$,
$\alpha\overline{\alpha}^T=1$,
we have
$$\alpha H \overline{\alpha}^T \geq |\alpha B \alpha^T|^2.$$
\end{lemma}
\begin{proof}
There is a decomposition for any complex symmetric matrix
$$B=U\mathrm{diag}[\lambda_1,...,\lambda_n]U^T,$$
where $U=[u_{ij}]$ is unitary, $\underset{i}{\sum} u_{ji}\overline{u}_{ji}=1$,  $\alpha=[a_1,...,a_n]$,
$\underset{i}{\sum} a_i\overline{a}_i=1$.
Because
$$H=B\overline{B}^T=U\mathrm{diag}[|\lambda_1|^2,...,|\lambda_n|^2]\overline{U}^T,$$
we only need to show that
$$(\alpha U) \mathrm{diag}(|\lambda_1|^2,...,|\lambda_n|^2)\overline{(\alpha U)}^T \geq |(\alpha U) \mathrm{diag}(\lambda_1,...,\lambda_n) (\alpha U)^T|^2.$$
Let $\beta=\alpha U=[b_1,...,b_n]$, then $\beta \overline{\beta}^T=1$, we need to show that
$$\beta \mathrm{diag}(|\lambda_1|^2,...,|\lambda_n|^2)\overline{\beta}^T \geq |\beta \mathrm{diag}(\lambda_1,...,\lambda_n) \beta^T|^2.$$
This is
$$\sum |\lambda_i|^2|b_i|^2\geq|\sum \lambda_ib^2_i|^2.$$
By Cauchy inequality and $\sum|b_i|^2=1$, we know the inequality is right
$$(\sum|\lambda_i|^2|b_i|^2)(\sum|b_i|^2)\geq (\sum |\lambda_i||b_i|^2)^2\geq|\sum \lambda_ib^2_i|^2.$$
\end{proof}

\subsection{$\varepsilon\geq w$}
\label{ss:varepsilon:geq:w}
In this section we establish
a relation between integrals of
eigenvalue spectrum of the curvature of the Hodge bundle over a Teichm\"uller curve
and its Harder--Narasimhan slopes.

Recall that by $\mu_i(E)$ we denote the Harder--Narasimhan
slope of a holomorphic vector bundle over a curve $C$, and by $w_i(E)$
we denote the corresponding normalized Harder--Narasimhan
slope, so that $\mu_i(E)=(\chi/2) w_i(E)$, where $-\chi(C)=2g-2+|\Delta|$
is the Euler characteristic of the underlying curve $C$. By
$\mathrm{d}\sigma$ we denote the volume element~\eqref{eq:d:sigma}
on $C$.

The following theorem is essentially contained in
Atiyah--Bott~\cite{AB82}:

\begin{theorem}\cite[p.573-575]{AB82}\label{ab}
Let $E$ be a Hermitian vector bundle of rank $n$ on a Riemann surface
$M$ with a volume form $\mathrm{d}\sigma$, and let
$$\Lambda_1(E)\geq...\geq\Lambda_n(E),$$
be the eigenvalues of $\frac{i}{2\pi}\cdot\Theta(E)$.
For $1\leq k\leq n$,
we have
$$\sum^k_{j=1} \int_M\Lambda_j(E)\mathrm{d}\sigma\geq \sum^k_{j=1}  \mu_j(E).$$
It is equality when $k=n$.
\end{theorem}
\begin{proof}
We shall begin by proving in the simple case when
$$\mu_1=\mu_2=...=\mu_r>\mu_{r+1}=...=\mu_n.$$
So that the Harder-Narasimahn filtration of the bundle $E$ has just two steps. We have an exact sequence of vector bundles
$$0\rightarrow D_1 \rightarrow E\rightarrow D_2 \rightarrow0$$
where $D_j$ has rank $m_j$ Chern class $k_j(j=1,2)$ so that $\mu_1=k_1/m_1$ and $\mu_n=k_2/m_2$. For convenience we shall use the notation $\mu^j=k_j/m_j(j=1,2)$. For the connection defined by the holomorphic structure and natural Hermitian metric. The curvature $\Theta(E)$ can then be written as the form
$$\Theta(E)=\begin{bmatrix}
    F_1-\eta\wedge\eta^* & d\eta \\
   -d\eta^*   & F_2-\eta^*\wedge\eta
  \end{bmatrix},
$$
where $F_j$ is the curvature of the metric connection of $D_j$, $\eta\in\Omega^{0,1}(M,Hom(D_2,D_1))$,$\eta^*$ is its transposed conjugate and $d\eta$ is the covariant differential. Now let $f_j,\alpha_j$ be scalar $m_j\times m_j$ matrices such that
$$\mathrm{trace}\,f_j=\mathrm{trace}\,*F_j$$
$$\mathrm{trace}\,\alpha_1=\mathrm{trace}\,*(\eta\wedge\eta^*)=-\mathrm{trace}\,*(\eta^*\wedge\eta)=-\mathrm{trace}\,\alpha_2.$$
We know that $\frac{i}{2\pi}*\Theta(E)$ is a Hermitian matrix. By the equivalence condition \ref{confun} of convexity, some elementary inequalities concerning convex invariant function $\phi$ show that
$$\phi(*\Theta(E))\geq\phi\begin{bmatrix}
    f_1-\alpha_1 & 0 \\
   0  & f_2-\alpha_2
  \end{bmatrix}.$$
 In particular by Lemma \ref{schur} it implies
  $$\Lambda_1(E)\geq \frac{i}{2\pi}\frac{\mathrm{trace}\,(f_1-\alpha_1)}{m_1}.$$
But the Chern class $k_j$ of $D_j$ is given by
$$k_j=\frac{i}{2\pi}\int_M \text{trace}\,f_j\mathrm{d}\sigma.$$
Since $f_j$ is scalar matrix this means that $\int_M \text{trace}f_j$ is scalar matrix whose diagonal entries are $-2\pi ik_j/m_j=-2\pi i \mu^j$. Also from (since $\eta\in \Omega^{0,1}$) it follows that $-i \text{trace} \alpha_1$ is non-negative and so
$$\int_M \alpha_1\mathrm{d}\sigma=2\pi i a_1,$$
where $a_1$ is non-negative scalar $m_1\times m_1$ matrix. Then
$$\int_M \alpha_2\mathrm{d}\sigma=2\pi i a_2,$$
where $a_2$ is non-negative scalar $m_2\times m_2$ matrix such that $\mathrm{trace}\,a_2=\mathrm{trace}\,a_1$. Hence we have
 $$ \frac{i}{2\pi}\int_M
  \begin{bmatrix}
  f_1-\alpha_1  & 0  \\
   0   &  f_2-\alpha_2
  \end{bmatrix}
  \mathrm{d}\sigma=[\mu+a],
$$
where $[\,]$ denotes the diagonal matrix defined by a vector, so that $[a]$ denotes the matrix $\begin{bmatrix}
       a_1 & 0          \\
       0    & a_2
\end{bmatrix}$.

But since $a_1\geq 0,a_2\leq0$ with $\mathrm{trace}\,a_1=-\mathrm{trace}\,a_2$ it follows easily that $\mu+a\geq\mu$ with respect to the partial ordering. Hence we have
$$ \int_M\Lambda_1(E)\mathrm{d}\sigma\geq  \frac{i}{2\pi}\int_M\frac{\mathrm{trace}\,(f_1-\alpha_1)}{m_1}\mathrm{d}\sigma\geq \frac{i}{2\pi}\int_M\frac{\mathrm{trace}\,f_1}{m_1}\mathrm{d}\sigma=\frac{k_1}{m_1}.$$
This completes the proof for the two-step case. The  general case proceeds in the same manner and we simply have to keep track of the notation. The details are as follows.

 We start with a holomorphic bundle $E$ with its canonical filtration of type $\mu$:
$$0=E_0\subset E_1\subset...\subset E_r= E,$$
where the quotients $D_j=E_j/E_{j-1}$ have normalized Chern classes $\mu^j$ with
$$\mu^1>\mu^2>...>\mu^r.$$
The curvature $\Theta(E)$ can then be expressed in a block form generalizing. For every $j<k$ we have an element
$$\eta_{jk}\in \Omega^{0,1}(M,Hom(D_k,D_j)),$$
so that $d\eta_{jk}$ appears in the $(j,k)$-block. The $\eta_{jk}$ are the components of the element
$$\eta_k\in \Omega^{0,1}(M,Hom(D_k,E_{k-1}))$$
related to the exact sequence
$$0\rightarrow E_{k-1}\rightarrow E_k\rightarrow D_k\rightarrow 0$$
Now define scalar non-negative $m_j\times m_j$ matrices $a_{jk}$ for $j<k$ by
$$\mathrm{trace}\,a_{jk}=\frac{1}{2\pi i}\int_M \mathrm{trace}\,(\eta_{jk}\wedge\eta^*_{jk})\mathrm{d}\sigma\geq 0,$$
and define $a_{kk}$ by
$$\mathrm{trace}\,a_{kk}=\frac{1}{2\pi i}\int_M\mathrm{trace}\,(\eta^*_k\wedge\eta_k)\mathrm{d}\sigma\leq 0,$$
so that $\underset{j\leq k}{\sum}\mathrm{trace}\,a_{jk}=0$. Then the convexity leads to the inequality
$$ \int_M\Lambda_1(E)\mathrm{d}\sigma\geq  \frac{i}{2\pi}\int_M\frac{\mathrm{trace}\,(f_1-\alpha_1)}{m_1}\mathrm{d}\sigma,$$
where $a$ stands for the vector (or diagonal matrix) whose $j$th block $a^j$ is the scalar (matrix)
$$a^j=\sum_{k\geq j}a_{jk}.$$
Equivalently the vector $a$ can be written as a sum
$$a=\sum b_k,$$
where $b_k$ is the vector corresponding to the diagonal matrix whose $j$th block is $a_{jk}$ for $j\leq k$(and zero for $j>k$). The fact that
$$\mathrm{trace}\,a_{jk}\geq 0\textit{ for }j<l\textit{ and } \sum_{j\leq k}\mathrm{trace}\,a_{jk}=0$$
implies that $b_k\geq 0$ relative to the partial ordering. Hence $a=\sum b_k\geq 0$ and so $\mu+a\geq \mu$. As before this then implies that
$$ \frac{i}{2\pi}\int_M\frac{\mathrm{trace}\,(f_1-\alpha_1)}{m_1}\mathrm{d}\sigma\geq \frac{i}{2\pi}\int_M\frac{\mathrm{trace}\,f_1}{m_1}\mathrm{d}\sigma=\frac{k_1}{m_1},$$
and so completes the general proof.
\end{proof}

Of course, the theorem implies the following corollary which was also noticed
earlier by M\"oller:
\begin{corollary}\label{kw}For a Teichm\"uller curve $C$, we have
   $$\varepsilon(C)\geq w(C).$$
\end{corollary}
\begin{proof}Because
   $$\mu_1(f_*\omega_{S/C})=\frac{\chi}{2},$$
by formulae~\eqref{ki}, \eqref{wi} and by Theorem~\ref{ab}, we have
   $$\sum^k_{j=1}\varepsilon_i(C)=\int_C\Lambda_j(\omega)\mathrm{d}\sigma/\frac{\chi}{2}\geq \sum^k_{j=1}\mu_i(f_*\omega_{S/C})/\mu_1(f_*\omega_{S/C})=\sum^k_{j=1}\mu_i(C).$$
\end{proof}

\subsection{$2\varepsilon\geq \lambda$}
\label{ss:varepsilon:geq:lambda}
The Lyapunov exponents of a vector bundle
endowed with a connection also can be viewed as logarithms of mean eigenvalues of
monodromy of the vector bundle along a flow on the base.

In the case of the Hodge bundle, we take a fiber of $H^1_{\mathbb{R}}$ and pull it along a
Teichm\"uller geodesic flow on the moduli space. We wait till the geodesic (or K\"ahler random walks \cite{Ko13}) winds a lot and comes close to the initial point and then compute the resulting monodromy
matrix $A(t)$. Finally, let $s_1(t)\geq...\geq s_{2g}(t)$ be the eigenvalues of $A^TA$, we compute logarithms of $s_i(t)$ and normalize
them by twice the length $t$ of the geodesic
$$\lambda_i=\underset{t\rightarrow \infty}{\mathrm{lim}}\frac{\mathrm{log}s_i(t)}{2t}.$$

By the Oseledets multiplicative ergodic
theorem, for almost all choices of initial data (starting point, starting direction) the
resulting $2g$ real numbers converge as $t \rightarrow \infty$, to limits which do not depend on the
initial data within an ergodic component of the flow. These limits $\lambda_1 \geq...\geq\lambda_{2g}$
are the Lyapunov exponents of the Hodge bundle along the Teichm\"uller geodesic flow. (\cite{EKZ11})

A simple corollary of Lemma \ref{schur} is
\begin{corollary}Let $a_1(t)\geq...\geq a_{2g}(t)$ be the diagonal elements of $A^TA$, then
$$\lambda_1\geq \underset{t\rightarrow \infty}{\mathrm{lim}}\mathrm{sup}\frac{\mathrm{log}\,a_1(t)}{2t}.$$
\end{corollary}

We confess that we do not know
any geometric interpretation of the quantity in the right hand
side of the latter inequality.

Forni has shown that the eigenvalues of curvature are closely related to Lyapunov exponents.
 Let $h(c)$ be the unique holomorphic form  such that $c$ is the cohomology class of the closed $1$-form $\mathrm{Re}\,h(c)$. By using the first variational formula \cite[p.19]{FMZ11}
$$\mathcal{L}\mathrm{log}||c||_{\omega}=-\frac{\mathrm{Re}B_{\omega}(h(c),h(c))}{||c||^2_{\omega}}\,,$$
(where $\mathcal{L}$ is the Lie derivative in direction $v=\omega^2$, and
$\|\cdot\|_\omega$ is the Hodge norm on $H^{1,0}(X)$)
he gets

\begin{corollary}\cite[Corollary 2.2]{Fo02}\label{fo1}
Let $\mu$ be any $\operatorname{SL}(2,\mathbb{R})$-invariant Borel probability ergodic
measure on the moduli space $\mathcal{H}_g$ of normalized Abelian differentials. The second Lyapunov exponent of the Kontsevich--Zorich cocycle with respect to the measure $\mu$, satisfies the following inequality
$$1> \int_{\mathcal{H}_g}\sqrt{|\Lambda_2(\omega)|}\mathrm{d}\mu(\omega)\geq\lambda^{\mu}_2.$$
\end{corollary}

Let $\{c_1,...,c_k\}$ be any Hodge-orthonormal basis of any isotropic subspace $I_k\subset H^1_{\mathbb{R}}$. His second variational formula \cite[p.21]{FMZ11} is
$$\Delta\mathrm{log}||c_1\wedge...\wedge c_k||_{\omega}=2\Phi_k(\omega,I_k)$$
here $w_i=h(c_i)$, $\Delta$ is the leafwise hyperbolic Laplacian
for the metric of curvature $-4$ on the Teichm\"uller leaves and
$$\Phi_k(\omega,I_k)=2\sum^{k}_{i=1}H_{\omega}(\omega_i,\omega_i)-\sum^{k}_{i,j=1}|B_{\omega}(\omega_i,\omega_j)|^2.$$

Then Theorem \ref{sumly} can be deduced from the following
\begin{corollary}\cite{Fo02}\cite[Corollary 3.2]{FMZ11}\label{fo2} Let $\mu$ be any $SL(2,\mathbb{R})$-invariant Borel probability ergodic
measure on the moduli space $\mathcal{H}_g$ of normalized Abelian differentials. Assume
that there exists $k\in {1,...,g-1}$ such that $\lambda^{\mu}_k> \lambda^{\mu}_{k+1}\geq 0$. Then the following formula holds:
$$\lambda^{\mu}_1+...+ \lambda^{\mu}_k=\int_{\mathcal{H}_g}\Phi_k(\omega,E^+_k(\omega))\mathrm{d}\mu(\omega).$$
\end{corollary}

We also give an upper bound of $\lambda_2$:
\begin{corollary}\label{kl}Let $\mu$ be any $SL(2,\mathbb{R})$-invariant Borel probability ergodic
measure on the moduli space $\mathcal{H}_g$ of normalized Abelian differentials. Then the following formula holds:
$$ 2\int_{\mathcal{H}_g}\Lambda_2(\omega)\mathrm{d}\mu(\omega)\geq\lambda^{\mu}_2.$$
In particular, for a Teichm\"uller curve
$$2\varepsilon_2\geq \lambda_2.$$
\end{corollary}
\begin{proof} For any $c\in \langle[\mathrm{Re}(\omega)], [\mathrm{Im}(\omega)]\rangle^{\bot}$, Lemma \ref{schur}implies
\begin{align}\label{el}2\Lambda_2(\omega)\geq \frac{2H_{\omega}(h(c),h(c))}{||c||^2_{\omega}} \geq \frac{2H_{\omega}(h(c),h(c))}{||c||^2_{\omega}}-\frac{|B_{\omega}(h(c),h(c))|^2}{||c||^4_{\omega}}.
\end{align}
Let $c$ be Kontsevich-Zorich cocycle with Lyapunov exponents $\lambda^{\mu}_2$. If $\lambda^{\mu}_2>\lambda^{\mu}_3$, by Corollary \ref{fo2}
$$2\int_{\mathcal{H}_g} \Lambda_2(\omega)\mathrm{d}\mu(\omega)\geq\int_{\mathcal{H}_g} \big(\frac{2H_{\omega}(h(c),h(c))}{||c||^2_{\omega}}-\frac{|B_{\omega}(h(c),h(c))|^2}{||c||^4_{\omega}}\big)\mathrm{d}\mu(\omega)=\lambda^{\mu}_2.$$
If $\lambda^{\mu}_2=...=\lambda^{\mu}_k>\lambda^{\mu}_{k+1}$, by Corollary \ref{fo2}, the result also can be deduced from
$$2(k-1)\int_{\mathcal{H}_g} \Lambda_2(\omega)\mathrm{d}\mu(\omega)\geq2\int_{\mathcal{H}_g} (\Lambda_2(\omega)+..+\Lambda_k(\omega))\mathrm{d}\mu(\omega)\geq\lambda^{\mu}_2+...+\lambda^{\mu}_k=(k-1)\lambda^{\mu}_2.$$

\end{proof}
The fiberwise inequality \ref{el} does not imply $\varepsilon_2\geq \lambda_2$. Because $\sqrt{|\Lambda_2(\omega)|}\geq 2\Lambda_2(\omega)$ if and only if $\frac{1}{4}\geq\Lambda_2(\omega)$, in contrast to Corollary \ref{fo1}, this Corollary is useful when $\varepsilon_2$ is small. Similarly we can get
$$2\sum^i_{j=2} \varepsilon_j\geq \sum^i_{j=2} \lambda_j \text{ for } i=2,...,n.$$
Or weak form
$$2\varepsilon(C)\geq \lambda(C).$$
(Warning: obviously $2\underset{j=1}{\overset{g}{\sum}} \varepsilon_j\neq 2\underset{j=1}{\overset{g}{\sum}} \lambda_j$. Here $\geq$ only means lying above.)

It will be interesting to
get the best inequality between $\varepsilon(C)$ and $\lambda(C)$.

\subsection{Hodge and Newton polygons}
\label{Hodge:and:Newton:polygons}
Since the Main Conjecture is originally inspired by
the Katz--Mazur theorem, we briefly recall this theorem.

Let $k$ be a finite field of $q=p^a$ elements; let $W$ denote its ring of Witt vectors, and $K$ the field of fractions of $W$. Let $X$ be projective and smooth over $W$, and such that the $W$-modules $H^r(X,\Omega^s_{X/W})$ are free (of rank $h^{s,r}$) for all $s,r$.

Form the polynomial
$$H_m(t)=\overset{m}{\underset{s=0}{\Pi}}(1-q^st)^{h^{s,m-s}}\in Z[t]\subset K[t],$$
which can be called the $m-$dimensional Hodge polynomial of $X/W$.
set
$$Z_m(t)=\mathrm{det}(1-F^a|_{H^m_{DR}(X/W)}t)\in K[t],$$
where $F$ is the canonical lifting of Frobenius on de Rham cohomology.

Now, for any polynomial of the form,
$R(t)=1+R_1t+R_2t^2+...+R_{\beta}t^{\beta}\in K[t].$
Mazur defined the polygon of $R(t)$ to be the convex closure in the Euclidean plane of the finite set points
$$(j,\mathrm{ord}_q(R_j)),j=0,1,...,\beta,$$
where $\mathrm{ord}_q(q)=1.$ The left most vertex of this polygon is the origin, while the right-most is $(\beta,\mathrm{ord}_q(R_\beta))$. The structure of this polygon is a measure of the $p$-adic valuations of the zeros of $R$.

According to our definition, the convex polygon is
$$(j,\mathrm{ord}_q(R_{\beta-j+1})),j=0,1,...,\beta.$$

Now what Katz conjectured and Mazur proved is

\begin{theorem}\cite{Ma72}\cite{Ma73}\label{km}The convex polygon of $H_m(t)$ (i.e.Hodge polygon) lies above (or on) the  convex polygon of $Z_m(t)$ (i.e.Newton polygon).
\end{theorem}
  A motivation to make the main conjecture is that we try to understand the sentence "Lyapunov exponents as Dynamical Hodge decomposition" (\cite{KZ97},\cite[p.37]{Zo06}).
\begin{corollary}\cite{Ma72}\cite{Ma73}If the Hodge numbers $h^{s,m-s}$ vanish for $0\leq s<t$, then the eigenvalues of $F^a$ acting on $H^m_{DR}(X/W)$ are divisible by $q^t$.
\end{corollary}
We think the Proposition \ref{zero} is an analogy of this Corollary.

\section{Conditional results and further conjectures} % Conjectures
\label{s:Conditional:results}

%----------------------------------------------------------------
\subsection{$\lambda\geq w?$}
\label{ss:lambda:ge:w}
Zorich asks the following question about the correspondence between Lyapunov exponents and characteristic numbers of some natural bundles. He sets this as the last problem of his survey.
\begin{problem}\cite[p.135]{Zo06}
Study individual Lyapunov exponents of the Teichm\"uller geodesic flow
\begin{itemize}
\item
for all known $SL(2;R)$-invariant subvarieties;
\item
for strata in large genera.
\end{itemize}
Are they related to characteristic numbers of some natural bundles over
appropriate compactifications of the strata?
\end{problem}
Originally inspired by Katz--Mazur Theorem \ref{km} in $p$-adic Hodge
theory, we state our Main Conjecture after checking over all
available numerical data\footnote{By the time the manuscript
was submitted to the journal, a proof of this conjecture was
announced by Eskin--Kontsevich--M\"oller--Zorich in~\cite{EKMZ}.}:
\begin{conjecture}\label{main}For any Teichm\"uller curve, we have
$$\lambda(C) \succeq w(C).$$
\end{conjecture}
That is
$$\left\{
  \begin{aligned}
    \sum^i_{j=1} \lambda_j\geq \sum^i_{j=1} w_j & \text{ for } i=1,...,g-1; \\
    \sum^g_{j=1} \lambda_j= \sum^g_{j=1} w_j    &.
  \end{aligned}
  \right.
.$$
Or, equivalently: $\overset{g}{\underset{j=i}{\sum}}\lambda_j\leq \overset{g}{\underset{j=i}{\sum}} w_j, \text{ for } i=2,...,g$; and $\overset{g}{\underset{j=1}{\sum}}\lambda_j=\overset{g}{\underset{j=1}{\sum}} w_j$.

\begin{remark}The result
$$\varepsilon(C)\succeq w(C)$$
in Corollary \ref{kw} has its root in
$$\Lambda_1(\omega)\geq H_{\omega}(\omega_i,\omega_i).$$
The result
$$2\varepsilon(C)\succeq \lambda(C)$$
in Corollary \ref{kl} has its root in
$$ 2\Lambda_1(\omega)\geq \frac{2H_{\omega}(h(c),h(c))}{||c||^2_{\omega}}-\frac{|B_{\omega}(h(c),h(c))|^2}{||c||^4_{\omega}}.$$
The Lemma \ref{hb} or \cite[p.22]{FMZ11} gives us
$$\frac{2H_{\omega}(h(c),h(c))}{||c||^2_{\omega}}-\frac{|B_{\omega}(h(c),h(c))|^2}{||c||^4_{\omega}}\geq \frac{H_{\omega}(h(c),h(c))}{||c||^2_{\omega}}.$$
We suspect that the Main Conjecture might have its root in the latter inequality.
\end{remark}

When the equality is attained, we also make the following
``rigidity'' conjecture:
\begin{conjecture}
If for some Teichm\"uller curve in the moduli space of Abelian differentials
of genus $g$ the following equality
$$\overset{k}{\underset{j=1}{\sum}}\lambda_j=\overset{k}{\underset{j=1}{\sum}}w_j$$
is achieved for some $k<g$ and
$$w_k\neq w_{k+1}$$
then the corresponding VHS contains
a rank $2(g-k)$ local subsystem.
\end{conjecture}

It can be considered as the inverse to Kontsevich--Forni's
formula~\ref{sumly}. We will illustrate an example
of a simple corollary of the two conjectures in
Proposition~\ref{zero}.

\begin{problem}[Lower Continuity Conjecture]
\label{pr:Continuity:Conjecture}
\label{continuous} We also hope that one can define
a $g$-vector $w(\mathcal{M})$
(of algebro-geometric origin) which would generalize
the normalized Harder--Narasimhan slope type from Teichm\"uller
curves to \textit{all} $\operatorname{GL}(2,\mathbb{R})$-invariant
submanifolds $\mathcal{M}$ in the moduli space of Abelian
differentials, and which would have the following natural
properties:
\begin{enumerate}
  \item
  \textit{If for each $i=1,\dots,g$ the normalized
  Harder--Narasimhan slope  $w_i(C)$ is constant for all
      Teichm\"uller curves $C$ in $\mathcal{M}$,
      then $w_i(\mathcal{M})=w_i(C)$;}
  \item
  \textit{If for some $i$, $1\le i\le g$, one has
  $w_i(C)\leq a$ for any Teichm\"uller curve $C$
  in $\mathcal{M}$, then $w_i(\mathcal{M})\leq a$.
  }
\end{enumerate}
\end{problem}

We say that a family of Teichm\"uller curves is dense inside the ambient
submanifold $\mathcal{N}$ of the moduli space of Abelian differentials if the closure
of the union of these curves coincides with $\mathcal{N}$. For example,
the arithmetic Teichm\"uller curves form a dense family in any
$\operatorname{GL}(2,\mathbb{R})$-invariant submanifold defined over
$\mathbb{Q}$, in particular, in any connected component of any stratum,
and in any hyperelliptic locus in any stratum.

\begin{condtheorem}
\label{th:condtheorem}
Consider a
$\operatorname{GL}(2,\mathbb{R})$-invariant submanifold $\mathcal{N}$
in the moduli space of Abelian differentials. Suppose that
$\mathcal{N}$ contains a dense family of Teichm\"uller curves and
that the normalized Harder--Narasimhan slopes $w_i(C)$ of all but at
most finite number of Teichm\"uller curves $C\subset\mathcal{N}$ in
this family satisfy the following inequality:
$$
\sum_{i=1}^k w_i(C) \ge a
$$
for some $k\le g$. The Main Conjecture implies that the Lyapunov
exponents $\lambda_i(\mathcal{N})$ of the Hodge bundle over the Teichm\"uller
geodesic flow on $\mathcal{N}$ satisfy the inequality
$$
\sum_{i=1}^k \lambda_i(\mathcal{N}) \ge a\,.
$$

\end{condtheorem}
\begin{proof}
The Main Conjecture implies that for any Teichm\"uller curve $C$
for which we have
$$
\sum_{i=1}^k w_i(C) \ge a
$$
we also have
$$
\sum_{i=1}^k \lambda_i(C) \ge \sum_{i=1}^k w_i(C) \ge a\,.
$$
The statement of the Theorem now follows from Theorem~\ref{EBW}.
\end{proof}

\begin{proof}[Proof of Corollary~\ref{cor:lambda:k:tends:to:1}]
By Theorem~\ref{YZ12a}, all normalized
Harder--Narasimhan slopes ${w_i(C)}$ are constant for all
Teichm\"uller curves  in the hyperelliptic connected components
$\mathcal{H}_g^{hyp}(2g-2)$ and  $\mathcal{H}_g^{hyp}(g-1,g-1)$, and the
corresponding Harder--Narasimhan types $w(C)$ are given by the
following formula:
$$
w(C) =
\begin{cases}
\Big(\frac{2g-1}{2g-1},\frac{2g-3}{2g-1},\dots,\frac{5}{2g-1},\frac{3}{2g-1},\frac{1}{2g-1}\Big)& \text{ for }C\subset \mathcal{H}_g^{hyp}(2g-2)\\
\Big(\frac{2g}{2g},\frac{2g-2}{2g},\dots,\frac{6}{2g},\frac{4}{2g},\frac{2}{2g}\Big)&  \text{ for }C\subset \mathcal{H}_g^{hyp}(g-1,g-1)\,.
\end{cases}
$$

The above expressions imply that for any fixed
$k\in\mathbb{N}$ one has $w_k(C)\to 1$ as $g\to+\infty$ and that
$$
\sum_{i=1}^k w_i(C) \ge a(g)\,,
$$
where $a(g)\to k$ as $g\to+\infty$. By
Conditional Theorem~\ref{th:condtheorem} the sum
$\lambda_1+\dots+\lambda_k$ of the top $k$
Lyapunov exponents of the hyperelliptic components of the strata
has the same asymptotic lower bound.

On the other hand, we have $1=\lambda_1\ge\lambda_i$
(actually, for $i>1$ the inequality is strict,
see~\cite{Zo96} and~\cite{Fo02}), so for any fixed $k>1$ we get
$$
k>\sum_{i=1}^k \lambda_i \ge a(g)\to k \text{ as }g\to+\infty\,,
$$
Taking into consideration that $1=\lambda_1>\lambda_2\ge\dots\ge\lambda_k$
this implies the statement of Corollary~\ref{cor:lambda:k:tends:to:1}.
\end{proof}

We complete Section~\ref{ss:lambda:ge:w} with two
more conditional statements in the spirit of
Corollary~\ref{cor:lambda:k:tends:to:1}. Though both results were
already proved by alternative (and quite involved) methods, they
serve as nice illustrations of potential further applications of the
the methods developed in this paper: the results immediately follow
from combination of the Main Conjecture and of Theorem~\ref{EBW}.

\begin{condtheorem}
\label{metEBW}
The assumption that the second Lyapunov exponent $\lambda_2$ of the
principal stratum $\mathcal{H}_3(1,1,1,1)$ of Abelian differentials in
genus $3$ is strictly greater than $\frac{1}{2}$ implies that algebraically
primitive Teichm\"uller curves are not dense in this stratum.
\end{condtheorem}
(Actually, the paper~\cite{BHM} contains an
alternative unconditional proof that there is at most a finite number
of algebraically primitive Teichm\"uller curves in this stratum.)
\begin{proof}
It is proved in~\cite{YZ12b} that
for any algebraically primitive Teichm\"uller curve $C$,
the non strict inequalities in the Main Conjecture become
equalities, namely, the following system of equalities is valid:
$$
w_i(C)=\lambda_i(C)\quad \text{ for }i=1,\dots,g\,.
$$
An explicit calculation (see Table~\ref{tab:genus:3})
shows that $w_2(C)\le\frac{1}{2}$
for any Teichm\"uller curve $C$ in the stratum
$\mathcal{H}_3(1,1,1,1)$. Thus, for any algebraically primitive
Teichm\"uller curve $C$
in $\mathcal{H}_3(1,1,1,1)$ we have
$$
\lambda_2(C)=w_2(C)\le\frac{1}{2}\,.
$$
Thus, by Theorem~\ref{EBW}, for any
$\operatorname{GL}(2,\mathbb{R})$-sumanifold
$\mathcal{N}\subseteq\mathcal{H}_3(1,1,1,1)$ in which the
algebraically-primitive Teichm\"uller curves are dense, we also have
$\lambda_2(\mathcal{N})\le\frac{1}{2}$, which implies that statement of the Theorem.
(Note that the experimental approximate value of the second Lyapunov exponent
$\lambda_2$ of the principal stratum in genus $3$ is
$\lambda_2(\mathcal{H}_3(1,1,1,1))\approx 0.5517$.)
\end{proof}

Clearly, the same method applies to any
stratum of Abelian differentials and to any
$\operatorname{GL}(2,\mathbb{R})$-sumanifold
$\mathcal{N}$ in it
for which one can find lower bounds for
some Lyapunov exponent $\lambda_i(\mathcal{N})$
exceeding the values of $w_i(C)$ for algebraically-primitive Teichm\"uller curves
in $\mathcal{N}$.

\begin{condtheorem}
Main Conjecture implies that the spectrum of Lyapunov
exponents of the Hodge bundle under the Teichm\"uller geodesic flow
on any Teichm\"uller curve
in the hyperelliptic component $\mathcal{H}_3^{hyp}(4)$,
and on the component $\mathcal{H}_3^{hyp}(4)$ itself, is simple,
$\lambda_1>\lambda_2>\lambda_3$.
\end{condtheorem}
\begin{proof}
The normalized Harder--Narasimhan slopes of all Teichm\"uller curves
in the hyperelliptic connected component $\mathcal{H}_3^{hyp}(4)$ are constant and
have values
$$w_1=1,w_2=3/5,w_3=1/5\,.$$
Thus, the Main Conjecture
implies that
$$\lambda_2\geq 3/5=w_2>w_3=1/5\geq\lambda_3\,,$$
and the Lyapunov spectrum of any Teichm\"uller curve in this component
is simple. Applying Conditional Theorem~\ref{th:condtheorem} we obtain the same inequalities
and, hence, the same conclusion for the entire component $\mathcal{H}_3^{hyp}(4)$.
\end{proof}

Note that simplicity of the spectrum of Lyapunov
exponents for any connected component of any stratum of Abelian
differentials was proved in~\cite{AV07}.

%-------------------------------------------------------------------
\subsection{Higgs fields}
\label{ss:Higgs:fields}
For the weight one $\mathbb{Q}$-VHS $(\mathbb{V},H^{1,0}=f_*\omega_{S/C}\subset H= (\mathbb{V}\otimes_{\mathbb{Q}}\mathcal{O}_{(C\backslash\Delta)})_{ext})$ which comes from the semi-stable family of curves $f\colon S\rightarrow C$.
The connection $\bigtriangledown$ composed with the inclusion and projection give a map
$$\theta^{1,0}\colon H^{1,0}\rightarrow H \rightarrow  H\otimes \Omega_C(\mathrm{log} \Delta) \rightarrow (H/H^{1,0})\otimes\Omega_C(\mathrm{log} \Delta),$$
which is $\mathcal{O}_C$-linear. If we extend $\theta^{1,0}$ by zero mapping to the associated graded sheaf we get a Higgs bunddle $(\mathrm{gr}(H),\theta)=(H^{1,0}\oplus H^{0,1},\theta^{1,0}\oplus 0)$. By definition this is a vector bundle on $C$ with a holomorphic map $\theta: F\rightarrow F\otimes \Omega_C(\mathrm{log} \Delta)$, the additional $\theta\wedge\theta$ being void if the base is a curve. (\cite{VZ04},\cite{Mo06})

Sub-Higgs bundles of a Higgs bundle $(F,\theta)$ are subbundles $G\subset F$, such that $\theta(G)\subset G$. The Higgs bundles is stable if for any sub-Higgs bundle $(G,\theta|_G)$
$$\frac{\mathrm{deg}(G)}{\mathrm{rk}(G)}<\frac{\mathrm{deg}(F)}{\mathrm{rk}(F)}.$$
Semi-stable is defined similarly but we allow now the weak inequality $\mathrm{deg}(G)/\mathrm{rk}(G)\leq \mathrm{deg}(F)/\mathrm{rk}(F).$

A Higgs bundle $(F,\theta)$ is polystable if
$$(F,\theta)=\bigoplus_i (F_i,\theta_i) \textit{ where } (F_i,\theta_i) \textit{ are stable Higgs bundles}.$$

Simpson shows that every stable Higgs bundle $(F,\theta)$ has a Hermitian-Yang-Mills metric. If $c_1(F)=0$, $\theta\wedge\theta=0$ and $c_2(F)[\omega]^{n-2}=0$ then the connection is flat. the last two condition is void if the base is a curve. Simpson also shows that for the complex variation of the Hodge structure $H$, $(\mathrm{gr}(H),\theta)$ is a polystable Higgs bundle such that each direct summand is of degree $0$. (\cite{Si88})

Simpson's correspondence allows us to switch back and forth between degree $0$ stable sub-Higgs bundles of $F$ and sub-local systems of $\mathbb{V}$.

The Higgs filed $(\mathrm{gr}(H),\theta)$ is the edge morphism
$$H^{1,0}=f_*\omega_{S/C}\rightarrow R^1f_*\mathcal{O}_S\otimes \Omega^1_C(\mathrm{log} \Delta)=H^{0,1}\otimes \Omega_C(\mathrm{log} \Delta)$$
of the tautological sequence
$$0\rightarrow f^*\Omega^1_C(\mathrm{log} \Delta)\rightarrow\Omega^1_S(\mathrm{log}(f^{-1}\Delta)) \rightarrow \Omega^1_{S/C}(\mathrm{log}(f^{-1}\Delta)) \rightarrow 0$$

By combining with some well know results of Higgs bundles (\cite{Ko87},\cite{VZ04}), we have the following property which says the non-uniformly
hyperbolic in dynamical systems \cite{Fo02} implies the positivity in algebraic geometry:
\begin{proposition}\label{zero}For any Teichm\"uller curve, we have:
$$\lambda_i>0\Rightarrow w_i>0.$$
If $w_k\neq 0,w_{k+1}=...=w_g=0$, then VHS contains a sub-local system of $\mathbb{V}$ of rank $2(n-k)$.
\end{proposition}
\begin{proof}
For any stable quotient bundle
$$H^{1,0}=f_*\omega_{S/C}\rightarrow V\rightarrow 0,$$
by dualization we have
$$0\rightarrow V^{\vee}\rightarrow f_*\omega_{S/C}^{\vee}=H^{0,1}.$$
Then we construct a stable sub-Higgs bundle:
$$(0\oplus V^{\vee},0)\subset (H^{1,0}\oplus H^{0,1},\theta^{1,0}\oplus 0).$$
Since $(\mathrm{gr}(H),\theta)$ is a polystable Higgs bundle, then this means that
$$-\mathrm{deg}(V)=\mathrm{deg}(0\oplus V^{\vee})\leq \mathrm{deg}(H^{1,0}\oplus H^{0,1})=0.$$
Denote by $HN_{min}(W)$ the last quotient in the Harder--Narasimhan filtration of a vector bundle $W$. Let $\mu_{min}(W)$ be the slope $\mu(HN_{min}(W))$. Then
$$\mu_{min}(f_*\omega_{S/C})\geq 0.$$
Because any quotient bundle $f_*\omega_{S/C}\overset{\varphi}{\rightarrow} Q\rightarrow 0$ induces a quotient bundle $f_*\omega_{S/C}\overset{\phi}{\rightarrow} HN_{min}(Q)\rightarrow 0$,
we have $\mu_{min}(Q)\geq \mu_{min}(f_*\omega_{S/C})$ (otherwise the map $\phi$ is zero).
We then obtain
 $$\mathrm{deg}(Q)\geq \mu_{min}(Q)\cdot \mathrm{rk}(Q)\geq \mu_{min}(f_*\omega_{S/C})\cdot \mathrm{rk}(Q)\geq 0.$$

If $w_k\neq 0$ and $w_{k+1}=...=w_g=0$, then $\mu_{min}(f_*\omega_{S/C})=0$. Since $(\mathrm{gr}(H),\theta)$ is a polystable Higgs bundle, then this means that
$$(0\oplus HN_{min}^{\vee}(f_*\omega_{S/C}),0)\textit{ is a direct summand of }(\mathrm{gr}(H),\theta).$$
So $HN_{min}^{\vee}(f_*\omega_{S/C})$ is a direct summand of $H^{0,1}$ and $HN_{min}(f_*\omega_{S/C})$ is a direct summand of $f_*\omega_{S/C}$. Furthermore \cite[p.1]{VZ04}
$$(HN_{min}(f_*\omega_{S/C})\oplus HN_{min}^{\vee}(f_*\omega_{S/C}),0)$$
is a polystable Higgs bundle such that each direct summand has degree $0$, so it comes from a sub-local system of $\mathbb{V}$ of rank $2(n-k)$.

By Theorem \ref{sumly}, we have:
$$0\leq\lambda_{k+1}+...+\lambda_{g}\leq2\mathrm{deg}(HN_{min}(f_*\omega_{S/C}))/(2g(C)-2+|\Delta|)=0$$
$$\Longrightarrow \lambda_{k+1}=...=\lambda_g=0.$$
\end{proof}
The reader can compare this Proposition with \cite[Theorem 3]{FMZ11}. It can be used to get some information of the zero eigenvalues of $EV(H_{\omega})$ and $EV(B_{\omega})$.

\subsection{Quadratic differentials}
For a Teichm\"{u}ller curve $C$ generated by $(Y,q)$ in
$\mathcal{Q}(d_1,...,d_s)$, let $(X,\omega)$ be the canonical double
covering. The curve $X$ comes with an involution $\tau$. Its
cohomology splits into the $\tau$-invariant and
$\tau$-anti-invariant part. Adapting the notation of \cite{EKZ11} we
let $g=g(Y)$ and $g_{eff}=g(X)-g$. Let $\lambda^+_i$ be the Lyapunov
exponents of the $\tau$-invariant part of $H^1(X,\mathbb{R})$ and
$\lambda^-_i$ be the Lyapunov exponents of the $\tau$-anti-invariant
part. The $\tau$-invariant part descends to $Y$ and hence the
$\lambda^+_i$ are the Lyapunov exponents of $(Y,q)$ we are primarily
interested in. Define
$$L^+=\lambda^+_1+...+\lambda^+_g$$
$$L^-=\lambda^-_1+...+\lambda^-_{eff}$$
The role of $\lambda^+_i$ is analogous to the ordinary sum of
Lyapunov exponents in the case of abelian differentials. We will
reprove the following formula in \cite[p.12]{KZ97}\cite[p.12]{EKZ11} by using double cover techniques \cite[p.236]{BHPV03}.
It is largely based on the work of Chen-M\"oller \cite{CM12}.
\begin{proposition}\label{quad}For a Teichm\"uller curve $C$  in
$\mathcal{Q}(d_1,...,d_s)$, we have
$$L^--L^+=\frac{1}{4}\sum_{\substack{
  j \mathrm{\, such \,that}\\
   d_j \mathrm{\,is \,odd}
  }} \frac{1}{d_j+2}.$$
\end{proposition}
\begin{proof}Note that $$\phi\colon \mathcal{Q}(...,d_i,...,d_j,...)\rightarrow \mathcal{H}_g(...,d_i/2,d_i/2,...,d_j+1,...)$$ for $d_i$ even and for $d_j$ odd. Since the double cover is branched at the singularities of odd order. Restrict this to a Teichm\"uller curve $C$ in $\mathcal{Q}(d_1,...,d_n)$. Then it gives rise to a Teichm\"uller curve isomorphic to $C$ in the corresponding stratum of abelian differentials. After suitable base change and compactification,
we can get two universal families $f', f$ and we have the following commutative diagram

$$\xymatrix{
    S' \ar[rr]^{\sigma} \ar[dr]_{f'}
                &  &    S \ar[dl]^{f}    \\
                & C                }
$$
and let $D'_j$ be the section of $f'\colon S'\rightarrow C$ over $D_j$ in case $d_j$ is odd and $D_{j,1},D_{j,2}$ be the sections over $D_j$ in case $d_j$ is even.

In the case when $d_j$ is odd, we have \cite[p.14]{CM12}
$$\sigma_*D'_j=D_j,\sigma^*D_j=2D'_j,$$
and the self-intersection number is
$$D^2_j=(\sigma_*D'_j)D_j=2(D'^2_j)=-\frac{\chi}{d_j+2}.$$
In the case when $d_j$ is even, we have
$$\sigma_*(D_{j,1}+D_{j,2})=2D_j,\sigma^*D_j=D_{j,1}+D_{j,2},$$
and the self-intersection number is
$$D^2_j=\frac{1}{2}(\sigma_*(D_{j,1}+D_{j,2}))D_j=\frac{1}{2}(D^2_{j,1}+D^2_{j,2})=-\frac{\chi}{d_j+2}.$$

The relative canonical bundle formula for the fibration $f':S'\rightarrow C$ is
$$\omega_{S'/C}=f'^*\mathcal{L}\otimes\mathcal{O}_{S'}\big(\sum_{\substack{
  j \mathrm{\, such \,that}\\
   d_j \mathrm{\,is \,even}
  }} \frac{d_j}{2}(D_{j,1}+D_{j,2})+\sum_{\substack{
  j \mathrm{\, such \,that}\\
   d_j \mathrm{\,is \,odd}
  }}(d_j+1)D'_j\big),$$
where $\mathcal{L}$ is the line bundle on $C$ corresponding to the generating abelian differential
and $\mathrm{deg}\mathcal{L}$ equals $\chi/2$.

The relative canonical bundle formula for the fibration $f\colon S\rightarrow C$ is
$$\omega^2_{S/C}=f^*\mathcal{F}\otimes\mathcal{O}_{S}(\sum d_jD_j),$$
where $\mathcal{F}$ is the line bundle on $C$ corresponding to the generating quadratic differential
and obviously $\mathcal{F}$ equals $\mathcal{L}^2$. (\cite[p.42]{CM12})

There is a smooth divisor $$B=\sum_{\substack{
  j \mathrm{\, such \,that}\\
   d_j \mathrm{\,is \,odd}
  }}   D_j$$
  on $S$, such that $B=2D$ for some effective divisor $D$. The double covering $\sigma\colon S' \rightarrow S$ is ramified exactly over $B$. We have
$$\omega_{S'}=\sigma^*(\omega_{S}\otimes \mathcal{O}_{S}(D))$$
and
$$\sigma_*(\mathcal{O}_{S'})=\mathcal{O}_{S}\oplus \mathcal{O}_{S}(-D).$$
The direct image of relative canonical bundle $f'_*\omega_{S'/C}$ decomposes into a direct sum
\begin{align*}
 f'_*\omega_{S'/C} &=f'_*\omega_{S'}\otimes \omega^{-1}_C=f_*\sigma_*\omega_{S'}\otimes \omega^{-1}_C  \\
 &=f_*(\sigma_*(\mathcal{O}_{S'})\otimes\omega_{S}\otimes \mathcal{O}_{S}(D))\otimes \omega^{-1}_C \\
 & =f_*(\omega_{S}\otimes \mathcal{O}_{S}(D)\oplus \omega_{S})\otimes \omega^{-1}_C\\
 &=f_*(\omega_{S/C})\oplus f_*(\omega_{S/C}+D).
\end{align*}

We assume $D\neq 0$ (when $D=0$, the result is trivial.). Then all the higher direct images of $\omega_{S/C}+D$ are zero. By Grothendieck-Riemann-Roch we have
\begin{align*}
ch(f_*(\omega_{S/C}+D)) &=f_*(ch(\omega_{S/C})\cdot ch(D)\cdot(1-\frac{\gamma}{2}+\frac{\gamma^2+\eta}{12}))\\
 &=f_*((1+\gamma+\frac{\gamma^2}{2})\cdot(1+D+\frac{D^2}{2})\cdot(1-\frac{\gamma}{2}+\frac{\gamma^2+\eta}{12})) \\
 &=f_*(1+(\frac{\gamma}{2}+D)+\frac{D(D+\gamma)}{2}+\frac{\gamma^2+\eta}{12}) \\
 &=\mathrm{rank}+(f_*\frac{D(D+\gamma)}{2}+\lambda),
\end{align*}
where $\gamma=c_1(\omega_{S/C})$, $\lambda=c_1(f_*\omega_{S/C})$ and $\eta$ the nodal locus in $f\colon S\rightarrow C$, they satisfy $\lambda=\frac{\gamma^2+\eta}{12}$ by Riemann-
Roch. Now we have
$$c_1(f_*(\omega_{S/C}+D))-c_1(f_*(\omega_{S/C}))=c_1(f_*(\omega_{S/C}+D))-\lambda=f_*\frac{D(D+\gamma)}{2}.$$
Because
\begin{align*}
\mathrm{deg}(f_*(\omega_{S/C}+D))-\mathrm{deg}(f_*(\omega_{S/C})) &=\frac{1}{2}D(\omega_{S/C}+D)\\
 &=\frac{1}{8}(\sum_{\substack{
  j \mathrm{\, such \,that}\\
   d_j \mathrm{\,is \,odd}
  }} D_j)(\sum_{\substack{
  j \mathrm{\, such \,that}\\
   d_j \mathrm{\,is \,odd}
  }} D_j+\sum d_jD_j+\mathcal{F}) \\
 &=\frac{1}{8}(\sum_{\substack{
  j \mathrm{\, such \,that}\\
   d_j \mathrm{\,is \,odd}
  }}(d_j+1) D^2_j+\chi) \\
 &=\frac{\chi}{2}\frac{1}{4}\sum_{\substack{
  j \mathrm{\, such \,that}\\
   d_j \mathrm{\,is \,odd}
  }}\frac{1}{d_j+2},
\end{align*}
and $f_*(\omega_{S/C})$ is $\sigma$-invariant part and $f_*(\omega_{S/C}+D)$  is $\sigma$-anti-invariant part. Hence we have
$$L^-=\mathrm{deg}(f_*(\omega_{S/C}+D)))/\frac{\chi}{2}, L^+=\mathrm{deg}(f_*(\omega_{S/C}))/\frac{\chi}{2}.$$
We get the formula
$$L^--L^+=\frac{1}{4}\sum_{\substack{
  j \mathrm{\, such \,that}\\
   d_j \mathrm{\,is \,odd}
  }}\frac{1}{d_i+2}.$$
\end{proof}

\begin{remark}
A version of Theorem~\ref{EBW} allows to
generalize the Proposition from
Teichm\"uller curves to
any $\operatorname{GL}^+(2,\mathbb{R})$-invariant submanifold in the
moduli space of meromorphic quadratic differentials with at most simple poles
defined over $\mathbb{Q}$, in particular, to every connected components of
each stratum.
\end{remark}

For the Teichm\"uller curve $C$ in $\overline{\mathcal{H}_g}(...,d_i/2,d_i/2,...,d_j+1,...)$, we have defined $g+g_{eff}$-vector
$$w(C)=(w_1,...,w_{g+g_{eff}}).$$
and moreover we know the upper bound of each $w_i$. Because
$$\mathrm{grad}(HN(f'_*\omega_{S'/C}))=\mathrm{grad}(HN(f_*(\omega_{S/C})))\oplus \mathrm{grad}(HN(f_*(\omega_{S/C}+D))),$$
we can divide $w_i$ into two parts
$$w^+_1,...,w^+_g;w^-_1,...,w^-_{eff}.$$
Here $w^+_1,...,w^+_g$ (resp.$w^-_1,...,w^-_{eff}$) come from the graded quotient $\mathrm{grad}(HN(f_*(\omega_{S/C})))$ (resp.$\mathrm{grad}(HN(f_*(\omega_{S/C}+D))$).

It is obvious by definition
$$L^+=w^+_1+...+w^+_g, L^-=w^-_1+...+w^-_{eff}.$$

\begin{example}
Consider the map $$\phi\colon\mathcal{Q}(1,2,-1,-1,-1)\rightarrow\mathcal{H}_3(2,1,1).$$
For a Teichm\"uller curve in $\mathcal{H}_3(2,1,1)$, we know that
$$w_1=1,w_2=1/2,w_3=1/3.$$
For a Teichm\"uller curve in $\mathcal{Q}(1,2,-1,-1,-1)$, we have
$$L^--L^+=5/6=1/4(1/3+1+1+1),L^-+L^+=w_1+w_2+w_3=11/6.$$
We get $L^-=4/3,L^+=1/2$, and so give us
$$w^+_1=1/2,w^-_1=1,w^-_2=1/3.$$
\end{example}

Of course we can  ask the same questions for $w^-_i$($w^+_i$) and $\lambda^-_i$($\lambda^+_i$).

\subsection*{Acknowledgements}
The article extends the author's report \cite{Yu14}
at the Oberwolfach conference ``Flat Surfaces and Dynamics on Moduli
Space''. The author thanks Mathematisches Forschungsinstitut at
Oberwolfach for hospitality and
Howard Masur, Martin M\"oller and Anton Zorich for
organization of the conference.

I am grateful to Alex Eskin who helped me to clarify
the proof of Theorem~\ref{th:condtheorem}, to
Chong~Song who stimulated me to read Atiyah--Bott's paper and to
Zhi~Hu for numerous discussions. I
would also like to thank Dawei~Chen, Yifei~Chen, Yitwah~Cheung,
Simion~Filip, Charles~Fougeron, Ronggang~Shi, Alex~Wright, Ze~Xu and
Lei~Zhang.

\section{Appendix}
We present here approximate numerical values on the individual
Lyapunov exponents $\lambda_i$ for all strata in genera $3$ and $4$
and of some strata in genus $5$. These values were computed experimentally
in~\cite{KZ97} and in~\cite{EKZ11}. The rational number representing
the sum $\overset{g}{\underset{j=1}{\sum}} \lambda_j$ of the positive
Lyapunov exponents in the right column of each table is exact; it is
computed rigorously in~\cite{EKZ11}. (Note that the sum contains as a
summand the top Lyapunov exponent $\lambda_1=1$ which is not present
in the table.)

We present also numerical data for the normalized slopes $w_i$ of the
Harder--Narasimhan filtration of the Hodge bundle over Teichm\"uller
curves in the corresponding strata computed rigorously
in~\cite{YZ12a} and in~\cite{YZ12b}) and reproduced in
Theorems~\ref{YZ12a} and~\ref{YZ12b} in the current paper.

When the numbers $w_i$ might vary from one
Teichm\"uller curve to another in the given stratum we provide
bounds~\eqref{eq:bounds:for:w} in the form of inequalities. When the
inequality sign is missing, it means that the Harder--Narasimhan type
$w(C)$ of any Teichm\"uller curve is constant for the stratum.

These data provides numerical evidence supporting the Main
Conjecture. We hope that it would be also useful for applications (like
evaluation of diffusion rates of some periodic polygonal billiards in
the plane).

%*************************************************
%\clearpage   % genus 3
%*************************************************

\begin{table}[hbt]
\small
$$
\begin{array}{|c|c||c|c||c|c||c|}

\hline
&&\multicolumn{2}{|c||}{}&
\multicolumn{2}{|c||}{}&
\\
\multicolumn{1}{|c|}{\text{Degrees}}&
\multicolumn{1}{|c||}{\text{Con-}}&
\multicolumn{2}{|c||}{\text{Lyapunov}}&
\multicolumn{2}{|c||}{\text{Normalized Harder-}}&
\multicolumn{1}{|c|}{\text{Sum}}
\\
\multicolumn{1}{|c|}{\text{of }}&
\multicolumn{1}{|c||}{\text{nected}}&
\multicolumn{2}{|c||}{\text{exponents}}&
\multicolumn{2}{|c||}{\text{Narasimhan slopes}}&
\multicolumn{1}{|c|}{\overset{g}{\underset{j=1}{\sum}} \lambda_j=}
\\
\cline{3-6}
\multicolumn{1}{|c|}{\text{zeros}}&
\multicolumn{1}{|c||}{\text{comp.}}&
\multicolumn{1}{|c|}{\lambda_2}&
\multicolumn{1}{|c||}{\lambda_3}&
\multicolumn{1}{|c|}{\quad\ w_2\quad\ }&
\multicolumn{1}{|c||}{w_3}&
\multicolumn{1}{|c|}{=\overset{g}{\underset{j=1}{\sum}} w_j}
\\
[-\halfbls] &&&&&&\\
\hline &&&&&& \\
[-\halfbls]
  (4)& hyp & 0.6156& 0.1844 & \frac{3}{5} & \frac{1}{5} & \frac{9}{5}  \\
[-\halfbls] &&&&&&\\
[-\halfbls] &&&&&&\\
\hline &&&&&& \\
[-\halfbls]
  (4) & odd & 0.4179 & 0.1821 & \frac{2}{5} & \frac{1}{5} & \frac{8}{5} \\
[-\halfbls] &&&&&&\\
[-\halfbls] &&&&&&\\
\hline &&&&&& \\
[-\halfbls]
  (3,1) &   & 0.5202 & 0.2298 & \frac{2}{4} & \frac{1}{4} & \frac{7}{4} \\
[-\halfbls] &&&&&&\\
[-\halfbls] &&&&&&\\
\hline &&&&&& \\
[-\halfbls]
  (2,2) & hyp & 0.6883 & 0.3117 & \frac{2}{3} & \frac{1}{3} & 2 \\
[-\halfbls] &&&&&&\\
[-\halfbls] &&&&&&\\
\hline &&&&&& \\
[-\halfbls]
  (2,2) & odd & 0.4218 & 0.2449 & \frac{1}{3} & \frac{1}{3} & \frac{5}{3} \\
[-\halfbls] &&&&&&\\
[-\halfbls] &&&&&&\\
\hline &&&&&& \\
[-\halfbls]
  (2,1,1) &   & 0.5397& 0.2936 & \frac{1}{2} & \frac{1}{3} & \frac{11}{6} \\
[-\halfbls] &&&&&&\\
[-\halfbls] &&&&&&\\
\hline &&&&&& \\
[-\halfbls]
  (1,1,1,1) &   &  0.5517 & 0.3411 & \leq 1/2  & \leq 1/2 & \frac{53}{28} \\
[-\halfbls] &&&&&&\\
\hline
\end{array}
$$
\caption{
\label{tab:genus:3}
All strata in genus $3$}
\end{table}

%*************************************************
% \clearpage   % genus 4
%*************************************************

\begin{table}
\small
$$
\begin{array}{|c|c||c|c|c||c|c|c||c|}

\hline
&&\multicolumn{3}{|c||}{}&
\multicolumn{3}{|c||}{}&
\\
\multicolumn{1}{|c|}{\text{Degrees}}&
\multicolumn{1}{|c||}{\text{Con-}}&
\multicolumn{3}{|c||}{\text{Lyapunov}}&
\multicolumn{3}{|c||}{\text{Normalized Harder-}}&
\multicolumn{1}{|c|}{\text{Sum}}
\\
\multicolumn{1}{|c|}{\text{of }}&
\multicolumn{1}{|c||}{\text{nected}}&
\multicolumn{3}{|c||}{\text{exponents}}&
\multicolumn{3}{|c||}{\text{Narasimhan slopes}}&
\multicolumn{1}{|c|}{\overset{g}{\underset{j=1}{\sum}} \lambda_j=}
\\
\cline{3-8}
\multicolumn{1}{|c|}{\text{zeros}}&
\multicolumn{1}{|c||}{\text{comp.}}&
\multicolumn{1}{|c|}{\lambda_2}&
\multicolumn{1}{|c|}{\lambda_3}&
\multicolumn{1}{|c||}{\lambda_4}&
\multicolumn{1}{|c|}{\,\ w_2\,\ }&
\multicolumn{1}{|c|}{\,\ w_3\,\ }&
\multicolumn{1}{|c||}{w_4}&
\multicolumn{1}{|c|}{=\overset{g}{\underset{j=1}{\sum}} w_j}
\\
[-\halfbls] &&&&&&&&\\
\hline &&&&&&&& \\
[-\halfbls]
(6) & hyp &
0.7375 & 0.4284 & 0.1198&
\frac{5}{7} & \frac{3}{7} & \frac{1}{7} &
\frac{16}{7}
\\
[-\halfbls] &&&&&&&&\\
\hline &&&&&&&& \\ [-\halfbls]
(6) & even &
0.5965 & 0.2924 & 0.1107 &
\frac{4}{7} & \frac{2}{7} & \frac{1}{7} &
\frac{14}{7}
\\
[-\halfbls] &&&&&&&&\\
\hline &&&&&&&& \\ [-\halfbls]
(6) & odd &
0.4733 & 0.2755 & 0.1084 &
\frac{3}{7} & \frac{2}{7} & \frac{1}{7} &
\frac{13}{7}
\\
[-\halfbls] &&&&&&&&\\
\hline &&&&&&&& \\ [-\halfbls]
(5, 1) &  &
0.5459 & 0.3246 & 0.1297 &
\frac{1}{2} & \frac{1}{3} & \frac{1}{6} &
2
\\
[-\halfbls] &&&&&&&&\\
\hline &&&&&&&& \\ [-\halfbls]
(3, 3) & hyp &
0.7726 & 0.5182 & 0.2097 &
\frac{3}{4} & \frac{2}{4} & \frac{1}{4} &
\frac{5}{2}
\\
[-\halfbls] &&&&&&&&\\
\hline &&&&&&&& \\ [-\halfbls]
(3, 3) & nonhyp &
0.5380 & 0.3124 & 0.1500 &
\frac{2}{4} & \frac{1}{4} & \frac{1}{4} &
2
\\
[-\halfbls] &&&&&&&&\\
\hline &&&&&&&& \\ [-\halfbls]
(4, 2) & even &
0.6310 & 0.3496 & 0.1527 &
\frac{3}{5} & \frac{1}{3} & \frac{1}{5} &
\frac{32}{15}
\\
[-\halfbls] &&&&&&&&\\
\hline &&&&&&&& \\ [-\halfbls]
(4, 2) & odd &
0.4789 & 0.3134 & 0.1412&
\frac{2}{5} & \frac{1}{3} & \frac{1}{5} &
\frac{29}{15}
\\
[-\halfbls] &&&&&&&&\\
\hline &&&&&&&& \\ [-\halfbls]
(2,2,2) & odd &
0.4826 & 0.3423 & 0.1749 &
\frac{1}{3} & \frac{1}{3} & \frac{1}{3} &
2
\\
[-\halfbls] &&&&&&&&\\
\hline &&&&&&&& \\ [-\halfbls]
(3,2,1) &  &
0.5558 & 0.3557 & 0.1718 &
\frac{1}{2} & \frac{1}{3} & \frac{1}{4} &
\frac{25}{12}
\\
[-\halfbls] &&&&&&&&\\
\hline &&&&&&&& \\ [-\halfbls]
(2,2,2) & even &
0.6420 & 0.3785 & 0.1928 &
\le\frac{2}{3} & \le\frac{1}{3} & \le\frac{1}{3} &
\frac{737}{333}
\\
[-\halfbls] &&&&&&&&\\
\hline &&&&&&&& \\ [-\halfbls]
(1,1,1,3) &  &
0.5600 & 0.3843 & 0.1849 &
\le\frac{1}{2} & \le\frac{1}{2} & \le\frac{1}{4} &
\frac{66}{31}
\\
[-\halfbls] &&&&&&&&\\
\hline &&&&&&&& \\ [-\halfbls]
(1,1,2,2) &  &
0.5604 & 0.3809 & 0.1982 &
\le\frac{2}{3} & \le\frac{1}{2} & \le\frac{1}{3} &
\frac{5045}{2358}
\\
[-\halfbls] &&&&&&&&\\
\hline &&&&&&&& \\ [-\halfbls]
(1,1,1,1,2) &  &
0.5632 & 0.4032 & 0.2168 &
\le\frac{1}{2} & \le\frac{1}{2} & \le\frac{1}{3} &
\frac{131}{60}
\\
[-\halfbls] &&&&&&&&\\
\hline &&&&&&&& \\ [-\halfbls]
(1,1,1,1,1,1) &  &
0.5652 & 0.4198 & 0.2403 &
\le\frac{1}{2} & \le\frac{1}{2} & \le\frac{1}{2} &
\frac{839}{377}
\\
[-\halfbls] &&&&&&&&\\ \hline
\end{array}
$$
\caption{All strata in genus $4$}
\end{table}

%*************************************************
\clearpage   % Genus 5
%*************************************************

\begin{table}
\small
$$
\begin{array}{|c|c||c|c|c|c||c|c|c|c||c|}

\hline
&&\multicolumn{4}{|c||}{}&
\multicolumn{4}{|c||}{}&\\

\multicolumn{1}{|c|}{\text{Degrees}}&
\multicolumn{1}{|c||}{\text{Con-}}&
\multicolumn{4}{|c||}{\text{Lyapunov}}&
\multicolumn{4}{|c||}{\text{Normalized Harder-}}&
\multicolumn{1}{|c|}{\text{Sum}}
\\

\multicolumn{1}{|c|}{\text{of }}&
\multicolumn{1}{|c||}{\text{nected}}&
\multicolumn{4}{|c||}{\text{exponents}}&
\multicolumn{4}{|c||}{\text{Narasimhan slopes}}&
\multicolumn{1}{|c|}{\overset{g}{\underset{j=1}{\sum}} \lambda_j=}
\\
\cline{3-10}

\multicolumn{1}{|c|}{\text{zeros}}&
\multicolumn{1}{|c||}{\text{comp.}}&
\multicolumn{1}{|c|}{\lambda_2}&
\multicolumn{1}{|c|}{\lambda_3}&
\multicolumn{1}{|c|}{\lambda_4}&
\multicolumn{1}{|c||}{\lambda_5}&
\multicolumn{1}{|c|}{w_2}&
\multicolumn{1}{|c|}{w_3}&
\multicolumn{1}{|c|}{w_4}&
\multicolumn{1}{|c||}{w_5}&
\multicolumn{1}{|c|}{=\overset{g}{\underset{j=1}{\sum}} w_j}
\\
[-\halfbls] &&&&&&&&&&\\
\hline &&&&&&&&&& \\ [-\halfbls] (8) & hyp &
0.799 & 0.586 & 0.306 & 0.087 &
\frac{7}{9} & \frac{5}{9} & \frac{3}{9} & \frac{1}{9} &
\frac{25}{9}
\\
[-\halfbls] &&&&&&&&&&\\
\hline &&&&&&&&&& \\ [-\halfbls]
(8) & even &
0.597 & 0.363 & 0.190 & 0.073 &
\frac{5}{9} & \frac{3}{9} & \frac{2}{9} & \frac{1}{9} &
\frac{20}{9}
\\
[-\halfbls] &&&&&&&&&&\\
\hline &&&&&&&&&& \\ [-\halfbls]
(8) & odd &
0.515 & 0.343 & 0.181 & 0.071 &
\frac{4}{9} & \frac{3}{9} & \frac{2}{9} & \frac{1}{9} &
\frac{19}{9}
\\
[-\halfbls] &&&&&&&&&&\\
\hline &&&&&&&&&& \\ [-\halfbls]
(6, 2) & odd &
0.521 & 0.369 & 0.212 & 0.089 &
\frac{3}{7} & \frac{1}{3} & \frac{2}{7} & \frac{1}{7}&
\frac{46}{21}
\\
[-\halfbls] &&&&&&&&&&\\
\hline &&&&&&&&&& \\ [-\halfbls]
(5, 3) & - &
0.562 & 0.376 & 0.216 & 0.096 &
\frac{1}{2} & \frac{1}{3} & \frac{1}{4} & \frac{1}{6} &
\frac{9}{4}
\\
[-\halfbls] &&&&&&&&&&\\
\hline &&&&&&&&&& \\ [-\halfbls]
(4, 4) & hyp &
0.819 & 0.639 & 0.390 & 0.152 &
\frac{4}{5} & \frac{3}{5} & \frac{2}{5} & \frac{1}{5} &
3
\\
[-\halfbls] &&&&&&&&&&\\
\hline &&&&&&&&&& \\ [-\halfbls]
(7, 1) & - &
0.560 & 0.378 & 0.207 & 0.082 &
\leq \frac{3}{4} & \leq \frac{1}{2} & \leq \frac{3}{8} & \leq \frac{1}{8}&
\frac{2423}{1088} \\
[-\halfbls] &&&&&&&&&&\\
\hline &&&&&&&&&& \\ [-\halfbls]
(6, 2) & even &
0.604 & 0.386 & 0.221 & 0.092 &
\leq \frac{5}{7} & \leq \frac{4}{7} & \leq \frac{1}{7} & \leq \frac{1}{7}&
\frac{178429}{77511}
\\
[-\halfbls] &&&&&&&&&&\\
\hline &&&&&&&&&& \\ [-\halfbls]
(6, 1, 1) & - &
0.563 & 0.397 & 0.230 & 0.094 &
\leq \frac{5}{7} & \leq \frac{1}{2} & \leq \frac{3}{7} & \leq \frac{1}{7}&
\frac{59332837}{25961866}
\\
[-\halfbls] &&&&&&&&&&\\
\hline &&&&&&&&&& \\ [-\halfbls]
(5, 2, 1) & - &
0.564 & 0.396 & 0.237 & 0.103 &
\leq \frac{2}{3} & \leq \frac{1}{2} & \leq \frac{1}{3} & \leq \frac{1}{6}&
\frac{4493}{1953}
\\
[-\halfbls] &&&&&&&&&&\\
\hline &&&&&&&&&& \\ [-\halfbls]
(5, 1, 1, 1) & - &
0.565 & 0.415 & 0.253 & 0.108 &
\leq \frac{2}{3} & \leq \frac{1}{2} & \leq \frac{1}{2} & \leq \frac{1}{6}&
\frac{103}{44}
\\
[-\halfbls] &&&&&&&&&&\\ \hline
\end{array}
$$
\caption{Some strata in genus $5$}
\end{table}

%*************************************************
% \clearpage
%*************************************************

\begin{example}\label{ex}
We illustrate the numerical evidence for the Main Conjecture taking
the connected component with odd parity of the spin structure
of the stratum $\mathcal{H}^{odd}_5(6,2)$ as an example.

$$\begin{array}{rcccl}
 \lambda_1  =&1&=&1&=  w_1, \\
 \lambda_1+\lambda_2\approx& 1.52&\geq& 10/7&=w_1+w_2, \\
 \lambda_1+\lambda_2+\lambda_3\approx& 1.89 &\geq& 37/21&=w_1+w_2+w_3,\\
 \lambda_1+\lambda_2+\lambda_3+\lambda_4\approx &2.1&\geq &43/21&=w_1+w_2+w_3+w_4,\\
 \lambda_1+\lambda_2+\lambda_3+\lambda_4+\lambda_5=&46/21&=&46/21&=w_1+w_2+w_3+w_4+w_5.
\end{array}$$
Or, equivalently,
$$\begin{array}{rcccl}
\lambda_1+\lambda_2+\lambda_3+\lambda_4+\lambda_5=&46/21&=&46/21&=w_1+w_2+w_3+w_4+w_5,\\
 \lambda_2+\lambda_3+\lambda_4+\lambda_5\approx &25/21&= &25/21&=w_2+w_3+w_4+w_5,\\
  \lambda_3+\lambda_4+\lambda_5\approx& 0.67 &\leq& 16/21&=w_3+w_4+w_5,\\
  \lambda_4+\lambda_5\approx& 0.30&\leq& 3/7&=w_4+w_5, \\
 \lambda_5  \approx&0.09&\leq&1/7&=  w_5.\\
 \end{array}$$
\end{example}

%*************************************************
\clearpage
%*************************************************

\end{document}